\documentclass[a4paper,12pt,oneside,reqno]{amsart}
\usepackage{hyperref}
\usepackage[headinclude,DIV13]{typearea}
\areaset{15.1cm}{25.0cm}
\usepackage{txfonts,amssymb,amsmath,amsthm,bbm}

\usepackage{graphicx, psfrag}
\usepackage{xcolor}
\usepackage{subfigure}

\newtheorem{theorem}{Theorem}[section]
\newtheorem{lemma}[theorem]{Lemma}

\definecolor{bbm}{RGB}{51,153,0}
\definecolor{above}{RGB}{128,0,128}
\definecolor{below}{RGB}{102,0,204}
\definecolor{cascade}{RGB}{204,0,0}
\definecolor{iid}{RGB}{153,51,0}

\theoremstyle{remark}
\newtheorem*{remark}{Remark}

\def\paragraph#1{\noindent \textbf{#1}}

\numberwithin{equation}{section}

\def\d{\mathrm{d}}

\def\<{\langle}
\def\>{\rangle}

\def\a{\alpha}
\def\b{\beta}
\def\e{\epsilon}

\def\g{\gamma}

\def\l{\lambda}

\def\o{\omega}

\def\G{\Gamma}

\def\del{\partial}
\def\R{{\Bbb R}}  
  
\def\P{{\Bbb P}}

\def\E{{\Bbb E}}

 \def \G {{\Gamma}}

 \def \b {{\beta}}

 \def \g {{\gamma}}
 \def \l {{\lambda}}
 \def \d {{\delta}}
 \def \a {{\alpha}}
 \def \o {{\omega}}

 \def \del {{\partial}}

 \def \ba {\begin{array}}
 \def \ea {\end{array}}

 \newcommand{\be}{\begin{equation}}
 \newcommand{\ee}{\end{equation}}

\newcommand{\bea}{\begin{eqnarray}}
 \newcommand{\eea}{\end{eqnarray}}
\def\TH(#1){\label{#1}}\def\thv(#1){\ref{#1}}
\def\Eq(#1){\label{#1}}\def\eqv(#1){(\ref{#1})}

\def\sfrac#1#2{{\textstyle{#1\over #2}}}

 \def \1{\mathbbm{1}}

\def \zet {{\mathfrak z}}

\def\eee{{\mathrm e}}

\def \log{\ln}
      
 \def\tb{\tilde \beta}
 \def\tg{\tilde\gamma}

\begin{document}

 \title[Speed of invasion]{The speed of invasion in an advancing population}
\author[A. Bovier]{Anton Bovier}
 \address{A. Bovier\\Institut f\"ur Angewandte Mathematik\\
University of Bonn\\ Endenicher Allee 60\\ 53115 Bonn, Germany }
\email{bovier@uni-bonn.de}
\author[L. Hartung]{Lisa Hartung}
 \address{L. Hartung\\
  Institut für Mathematik \\Johannes Gutenberg-Universität Mainz\\
Staudingerweg 9,
55099 Mainz, Germany}
\email{lhartung@uni-mainz.de}

\date{\today}

 \begin{abstract}  
 We derive rigorous estimates on the speed of invasion of an advantageous trait in a spatially advancing population in the context of a system of one-dimensional
  F-KPP equations. The model was introduced and studied heuristically and numerically in a paper by Venegas-Ortiz et al \cite{VAE2014}. In that paper, it was noted that the speed of invasion by the mutant trait is faster when the resident population is expanding in space compared to the speed when the resident population is already present everywhere. We use the Feynman-Kac representation to provide rigorous estimates that confirm these predictions.   
 \end{abstract}

\thanks{
This work was partly funded by the Deutsche Forschungsgemeinschaft (DFG, German Research Foundation) under Germany's Excellence Strategy - GZ 2047/1, Projekt-ID 390685813 and GZ 2151 - Project-ID 390873048,
through Project-ID 211504053 - SFB 1060, through Project-ID 233630050 -TRR 146, through Project-ID 443891315  within SPP 2265, and Project-ID 446173099. The thank the Club Santa Rosa on Lanzarote for providing the excellent environment for writing the bulk of this paper. 
 Data sharing not applicable to this article as no datasets were generated or analysed during the current study.
 }


 \maketitle

\section{Introduction} 
The present paper is motivated by an interesting paper by Venegas-Ortiz, Allen, and Evans \cite{VAE2014} that investigates the invasion of a spatially expanding population by a new trait. The classical model for the invasion of a gene in a spatially extended population \cite{fisher37} or the expansion of a population in space \cite{kpp} is the 
Fisher-Kolmogorov-Petrovsky-Piscounov (F-KPP) equation, that has been the subject of
intense investigation for over 80 years\footnote{A superficial search for ``KPP" in MathSciNet finds over 500 entries since 1967 alone.}. The F-KPP equation is a non-linear 
reaction-diffusion equation that admits travelling wave solutions to which solutions starting with suitable initial conditions converge. This has been known since the early work of 
Kolmogorov et al.\ \cite{kpp}, but has been made both more precise and more general in the seminal book  by Bramson \cite{B_C}. 

The model discussed in \cite{VAE2014} is a system of two coupled equations of the F-KPP type that describes the evolution of a population of two types (traits, alleles, ..) that 
diffuse, compete, and and switch between types. More specifically, they propose the system of equations
 \bea\Eq(fkpp.1)
  \del_t N_A&=&\frac 12 \del_{xx} N_A +\a N_A(K-N_A-N_B)-\b N_A+\g N_AN_B,\\
  \del_t N_B&=&\frac 12 \del_{xx} N_B +\a N_B(K-N_A-N_B)+\b N_A-\g N_AN_B.
 \Eq(fkpp.2)
 \eea
 $N_A,N_B$ represent the masses of traits $A$ and $B$, $K$ is the carrying capacity,
 $\a,\b,\g$ are parameters that satisfy
 \be
 \Eq(fkpp.3)
 \a>\g>\b/K\geq 0.
 \ee
 The different terms in these equations correspond to the following biological 
mechanisms:
\begin{itemize} 
\item[(i)] The terms $\del_{xx} N$ model the spatial diffusion of the population. Note that the diffusion coefficients are the same for both types. This can be seen as biologically plausible, but this choice is mainly done to simplify the mathematical treatment.
\item [(ii)] The  terms proportional to $\a$ describe logistic growth with the quadratic terms corresponding to competitive pressure. Again it is assumed that the pressure exerted by both types and on each type are the same. This again simplifies the 
mathematics.
\item[(iii)] The linear terms $\pm \b N_A$ can be interpreted as mutation rates from the 
$A$ population to the $B$ population. There effect is a net disadvantage of the $A$ population.
\item[(iv)] The non-linear terms $\pm \g N_AN_B$ are interpreted as horizontal gene transfer from the $B$-types to the $A$-types. The idea is that when an $A$ individual
encounters a $B$ individual, the genotype of the $B$ individual can be switched to the
$A$-type. 
\end{itemize}
The choice of parameters in \eqv(fkpp.3) ensures that the a priori disadvantaged $A$ type can reemerge in a developed $B$-population and a stable equilibrium 
with co-existing types exists. The question addressed  in \cite{VAE2014} is to analyse
how this effect leads to a hitchhiking of the $A$-type when the $B$-type is 
spreading in space.
  The authors of 
\cite{VAE2014} make the following interesting and somewhat surprising observation.
There are two easily derived travelling waves in the system. First, a population made purely of $B$ individuals remains in that state and advances with a speed $v_B$.
Second, if $B$ has invaded all space, and a $A$ population is introduced, there is 
(with the choice of parameters that ensures the instability of the $B$ population against the invasion of $A$ individuals) a travelling wave of $A$ particles that advances in the background of $B$ particles with a 
speed $v_A<v_B$. If, however, one starts with initial conditions where $A$ and $B$ particles are present, say in the negative half-line, then the $B$ population advances 
with speed $v_B$ again, but in some parameter range the $A$ population advances with a speed $v_c$ that is
strictly larger than the speed $v_A$ (and smaller than $v_B$). Somehow, the $A$ individuals sense the empty space ahead of the $B$-wave and get attracted 
to it. 

Venegas-Ortiz et al \cite{VAE2014} derive this result, and precise formulas for the speeds,
using local linearisation and matching of solutions. These findings are supported by numerical 
simulations. In the present paper we derive rigorous
estimates on the speeds using the Feynman-Kac representation, originally 
employed by Bramson \cite{B_C} to control the precise speed of convergence to the travelling wave in the original F-KPP equation. It turns out that this point of 
view not only allows to  give rigorous and precise bounds on the solutions of 
the system of equations, and hence the speeds, 
but also provides a clear and intuitive explanation for the fact that the empty space ahead
of the $B$-wave allows for a faster advance of the $A$-wave. Namely, we will see that 
this is driven by large, unlikely excursion of the Brownian motion  in the Feynman-Kac 
formula that reach ahead of the front of the $B$-wave. Mathematically, this involves some
delicate estimates  on probabilities of large excursions of Brownian bridges.

Systems of coupled F-KPP  equations have   been studied in different contexts in the literature, see e.g.  \cite{GL2019,CHTWW95,Holz2012,HS2012,HS2014,Holz2016,faye2018,keenan2021,LamYu21}. 
In particular, an analogous result to that in \cite{VAE2014} and the present paper was 
derived rigorously in \cite {HS2014}  using analytic methods.
Rather recently, there has been interest in such systems in the context of dormancy, see e.g. \cite{BHN22}. Applicable tools depend on the details of the equations. 
\cite{GL2019} use purely analytic methods involving sub- and super-solutions, while the equations appearing in \cite{CHTWW95} and
\cite{BHN22} allow for a representation in terms of branching Brownian motion and the use of martingale methods. The equations in 
\cite{VAE2014} (and \cite{HS2014,Holz2016,faye2018})  are particularly nice, as they allow for the use of the Feynman-Kac representation. However, even the introduction of 
two different diffusion constants seems to spoil this feature, and it seems unclear (albeit interesting) to see how this method can be extended to more general settings.

\paragraph{Outline.} The remainder of this paper is organised as follows. In Section 2 we give a precise
formulation of the model put forward in \cite{VAE2014}  and explain the special structure of the system that effectively reduces the problem to a time-dependent one-dimensional 
F-KPP equation. Afterwards we state our main result. Along the way we also recall some background on the standard F-KPP equation that will be needed. In Section 3 we present the Feynman-Kac representation, derive some first bounds, and give a heuristic explanation of the main result, based on the Feynman-Kac representation. Section 4 provides the necessary upper and lower bounds on the excursions of Brownian bridges. We compute fairly sharp 
bounds on the Laplace transforms of these excursions using the Laplace method.
Armed with these estimates, we derive upper and lower bounds on solutions from which 
the wave speed $v_c$ is inferred in Section 5. At the end of the paper, in  Section 6,
we discuss our results and point to possible future extensions.

\section{The F-KPP equations}
 
 It is convenient to introduce the total population mass $N_T\equiv N_A+N_B$
 and to write the equations \eqv(fkpp.1) and \eqv(fkpp.2)  in the the form 
 \bea\Eq(fkpp.4)
  \del_t N_T&=&\frac 12 \del_{xx} N_T +\a N_T(K-N_T),\\
 \del_t N_A&=&\frac 12 \del_{xx} N_A +\a N_A(K-N_T)-\b N_A+\g N_A(N_T-N_A).
 \Eq(fkpp.5)
 \eea
 We see that $N_T$ satisfies an autonomous F-KPP equation. Effectively, the second equation is a F-KPP equation with time dependent reaction rates. This structure is crucial for our analysis building on the Feynman-Kac formula. Equations of a similar structure have been also been studied in \cite{Holz2016,faye2018}. It is furthermore convenient to eliminate the parameters $K$ and $\a$ by rescaling.
 We define
 \bea
\Eq(fkpp.6)
v(t,x)&\equiv&\frac 1K N_T(t/(\a K), x/\sqrt{\a K}),\\
w(t,x)&\equiv&\frac 1K N_A(t/(\a K), x/\sqrt{\a K}).\Eq(fkpp.6.1)
\eea
Then $v$ and $w$ solve
\bea\Eq(fkpp.7)
\del_t v&=&\frac 12\del_{xx} v+v(1-v),
\\
\Eq(fkpp.8)
\del_t w&=&\frac 12\del_{xx} w +\left(1-\tb-(1-\tg)v-\tg w\right)w.
\eea
where $\tb=\b/(\a K)$ and $\tg=\g/\a$. Note that $1>\tg>\tb>0$.

Note that the system of equations has four spatially constant fixpoints:
\begin{itemize}
\item[(i)] $v=0, w=0$,
\item[(ii)] $v=0, w=(1-\tb)/\tg$,
\item[(iii)] $v=1,w=0$,
\item[(iv)] $v=1,w=1-\tb/\tg$.
\end{itemize}
The fixpoint (ii) is unphysical, since it corresponds to a negative mass for the population $B$. The fixpoints (i) and (iii) are unstable, and (iv) is the stable fixpoint.

The behaviour of $v$ is well-known from Bramson's work \cite{B_C}, so solving for 
$w$ amounts to solve the F-KPP equation with time dependent coefficients. A particularly simple situation arises if we choose initial  conditions such that 
$v(0,x)=1$, for all $x\in \R$. In that case $w$ solves the F-KPP equation 
\be
\Eq(fkpp.10)
\del_t w=\frac 12\del_{xx} w +\left(\tg-\tb-\tg w\right)w.
\ee
In this case, with suitable initial conditions (e.g. Heaviside), $w$ converges to a travelling wave
solution that moves with speed $\sqrt {2(\tg-\tb)}$. A more interesting situation arises 
if the initial conditions are such that $v(0,x)$ decays rapidly at $+\infty$ 
 and $w(0,x)$ is
non-zero. In that case, \cite{VAE2014} observed that 
the $w$-wave follows behind the $v$-wave, but moves faster than it would in 
a fully established population. 
Recall that the standard F-KPP equation \eqv(fkpp.7) admits travelling wave solutions
\be
v(t,x+\l t)=\o(x),
\Eq(fkpp.11)
\ee
where $\o$ solves the ode
\be\Eq(fkpp.12)
\frac 12\del_{xx} \o +\l\del_x \o +\o(1-\o)=0,
\ee
for all speeds $\geq \sqrt 2$.  
It was shown by Kolmogorov \cite{kpp} that \eqv(fkpp.12) has a unique solution up to 
translations such that $\lim_{x\downarrow-\infty}\o(x)=1$ and 
 $\lim_{x\uparrow\infty}\o(x)=0$. 
 We are only interested in the case $\l=\sqrt 2$, since solutions with initial condition that 
 converge rapidly to zero at infinity, and in particular with Heaviside initial conditions, 
 converge to travelling waves with this speed (see \cite{B_C} for more details).  
  
 We pick the solution for which $\o(0)=1/2$. 
 Lalley and Sellke \cite{LS} derived the probabilistic representation 
 \be
\Eq(fkpp.13)
1-\o(x)= \E\left[\eee^{-ZC\eee^{-\sqrt{2\pi}x}}\right],
\ee 
where $Z$ is a random variable, the limit of the so-called  \emph{derivative martingale},
and $C$ is a constant such that
\be
 \E\left[\eee^{-ZC}\right]=\frac 12.
\ee
Clearly, if $v$ solves \eqv(fkpp.7) with initial condition 
$v(0,x) = \o(x+a)$, then $v(t,x)=\o(x+a-\sqrt 2 t)$.
It is known that
\be
\Eq(fkpp.14) 
\o(x) \sim C x \eee^{-\sqrt 2 x},\quad \text{as}\;  x\uparrow +\infty,
\ee
and
\be
\Eq(fkpp.15) 
\o(x) \sim 1- c  \eee^{(2-\sqrt  2) x},\quad \text{as}\;  x\downarrow -\infty.
\ee
(\eqv(fkpp.14) is due to Bramson, \eqv(fkpp.15) is proven in the first arXiv version 
of \cite{ABK_G}). 
Bramson has shown that for any initial conditions that decay faster than $\eee^{-\sqrt 2x}$ at $+\infty$, 
\be
\Eq(fkpp.16)
v(t,x-m(t)) \rightarrow \o(x),
\ee
uniformly in $x$, as $t\uparrow\infty$, where 
\be\Eq(fkpp.17)
m(t)=\sqrt 2 t-\frac 3{2\sqrt 2}\ln t.
\ee
It will be convenient to analyse the system \eqv(fkpp.7), \eqv(fkpp.8) with initial conditions
$v(0,x)=\o(x+a)$ and $w(0,x) =(1-\tb/\tg) \1_{x\leq 0}$. With this choice, our problem is reduced
to studying the scalar equation 
\be
\Eq(fkpp.18)
\del_t w(t,x)=\frac 12\del_{xx} w(t,x) +\left(1-\tb-(1-\tg)\o\left(x+a-\sqrt  2t\right)-\tg w(t,x)
\right)w(t,x),
\ee
with initial condition $w(0,x) =(1-\tb/\tg) \1_{x\leq 0}$.

Our main result 
is the following.
\begin{theorem} \TH(main)
Let $a\in\R_+$. Let
\be\Eq(lisa.lan23)
u_c\equiv \max\left(\sqrt{2}-\frac{\tb}{\sqrt{2}\tg} \left(1+\sqrt{1-\tg}\right),\sqrt{2\left(\tg-\tb\right)}\right)
\ee
Then for all $\delta>0$ sufficiently small there exist constants $C_1,C_2>0$ such that
\be\Eq(lisa.lan24)
w(t, u_ct-C_1\log t)>\delta 
\ee
and
\be\Eq(lisa.lan25)
w(t, u_ct+C_2\log t+z)<1/t, \quad \forall z>0,
\ee
for all $t$ large enough.
\end{theorem}

\begin{remark}
Note that $u_c$ is strictly  larger than $\sqrt{2\left(\tg-\tb\right)}$ for $\tb$ small enough.
{Notice that Venegas-Ortiz et al.  derive in \cite{VAE2014} a rather complicated looking equation, (Eq. 8), and a simpler one (Eq. 9), obtained by expanding in $\tb$. Our results show that the second version is exact, provided $\tb $ is such that 
\be
 \sqrt{2}-\frac{\tb}{\sqrt{2}\tg} \left(1+\sqrt{1-\tg}\right)\geq \sqrt{2\left(\tg-\tb\right)},
 \ee
  while the first seems incorrect. This is 
also in agreement with the finding in \cite{GL2019}.}
An analogous result on an accelerated speed in a slightly different system of equations was derived by purely analytic methods by Holzer and Scheel \cite{HS2014}, Lemma 11.
\end{remark}
\begin{remark}
Note that in fact the result of Theorem \thv(main) does not depend on the choice of $a$ in the initial condition. This is not surprising as a finite shift of the initial condition 
does not affect the large time asymptotic of the solutions.
\end{remark}
 The remainder of this paper is devoted to proving Theorem \thv(main). In the process, we will derive precise bounds on the behaviour of the solutions.

\section{The Feynman-Kac representation}
 Bramson's analysis of the F-KPP equation \cite{B_C} is based on the Feynman-Kac representation. We will do the same for the equation 
 \eqv(fkpp.18). 
 
 \subsection{The representation and elementary bounds}

 \begin{lemma}\TH(fk.1) The solution of \eqv(fkpp.18) satisfies the equation
 \be\Eq(fk.2)
 w(t,x)=
 \E_x\left[\exp\left(\int_0^t 
 \left(1-\tb-(1-\tg)\o\left(B_s+a-\sqrt  2(t-s)\right)-\tg w(t-s, B_s)
\right)ds\right)w(0,B_t)\right],
\ee
where $B$ is a Brownian motion starting in $x$.
\end{lemma}
\begin{proof}
The proof is identical to the one in \cite{B_C}. 
\end{proof}

It is convenient to express the Brownian motion $B$ in terms of its endpoint $B_t$ and  a Brownian bridge 
\be\Eq(fk.3)
\zet_{x,B_t}^t(s)  = x+\frac{s}{t} (B_t-x) + \zet_{0,0}^t,
\ee
from $x$ to $B_t$.
Here $\zet_{0,0}^t$ is a Brownian bridge from $0$ to $0$ in time $t$. 
 Note that the bridge is independent of $B_t$. 
 This leads to the following reformulation of \eqv(fk.2).
 
 \begin{lemma}\TH(fk.4) The solution of \eqv(fkpp.18) satisfies
 \bea\Eq(fk.5)\nonumber
 w(t,x)&=& \frac{1}{\sqrt{2 \pi t}}\int_{-\infty}^\infty dy \eee^{-\frac{(x-y)^2}{2t}} w(0,y)\\\nonumber
&\times& \E\left[\exp\left(\int_0^t 
 \left(1-\tb-(1-\tg)\o\left(\zet^t_{x,y}(s)+a-\sqrt  2(t-s)\right)-\tg w\left(t-s, \zet_{x,y}^t(s)\right)
\right)ds\right)\right]\\\nonumber
&=& \frac{1}{\sqrt{2 \pi t}}\int_{-\infty}^\infty dy \eee^{-\frac{(x-y)^2}{2t}} w(0,y)\\\nonumber
&\times& \E\Biggl[\exp\Biggl(\int_0^t 
 \biggl(1-\tb-(1-\tg)\o\left(x\sfrac {t-s}t +\sfrac st y +\zet^t_{0,0}(s)+a-\sqrt  2(t-s)\right)\\
 &&-\tg w\left(t-s, x\sfrac {t-s}t +\sfrac st y
 +\zet_{0,0}^t(s)\right)
\biggr)ds\Biggr)\Biggr],
\eea
where $\E$ now refers to the expectation with respect to the Brownian bridges $\zet_{x,y}^t$ resp. $\zet_{0,0}^t$.
\end{lemma}

\begin{proof} Elementary.\end{proof}

The fact that $0\leq \o\leq 1$ and $0\leq w\leq 1-\tg/\tb$ yields the first bounds.

\begin{lemma}\TH(lwmma-fk.5) The solution of \eqv(fkpp.18) satisfies
\be\Eq(fk.6)
 w(t,x)\leq  \frac{1}{\sqrt{2 \pi t}}\int_{-\infty}^\infty dy \eee^{-\frac{(x-y)^2}{2t}} w(0,y)
 \eee^{(1-\tb)t}.
 \ee
 and 
\be\Eq(fk.6.1)
 w(t,x)\geq  \frac{1}{\sqrt{2 \pi t}}\int_{-\infty}^\infty dy \eee^{-\frac{(x-y)^2}{2t}} w(0,y).
 \ee
 For Heaviside initial conditions, this implies 
 \be
\Eq(fk.7)
\sqrt {\frac t{2\pi}} \frac {\eee^{-\frac{x^2}{2t}}}{x}\left(1-O (t/ x^2)\right)\leq \frac{w(t,x)}{(1-\tb/\tg)} \leq
\sqrt {\frac t{2\pi}} \frac {\eee^{-\frac{x^2}{2t}+(1-\tb)t}}{x}.
\ee
\end{lemma}
\begin{proof}
Eqs. \eqv(fk.6) and \eqv(fk.6.1) are immediate from the bounds on $\o$ and $w$ mentioned above. \eqv(fk.7) follows from the 
standard Gaussian tail estimates, see, e.g. \cite{LLR}.
\end{proof}
\subsection{First heuristics.} 
Since the term involving $\o$ is explicit, we can improve the upper bound \eqv(fk.6) as follows.
\bea\Eq(fk.8)
w(t,x) &\leq& 
 \sfrac{(1-\tb/\tg)}{\sqrt{2 \pi t}}\int_{-\infty} dy \eee^{-\frac{(x-y)^2}{2t}} \\\nonumber
&\times&
 \E\Biggl[\exp\Biggl(\int_0^t 
 \biggl(1-\tb-(1-\tg)\o\left(x\sfrac {t-s}t +\sfrac st y +\zet^t_{0,0}(s)+a-\sqrt  2(t-s)\right)
\biggr)ds\Biggr)\Biggr].
\eea
Since $w\leq  \o$, we also have the lower bound
\bea\Eq(fk.9)
w(t,x) &\geq& 
 \sfrac{(1-\tb/\tg)}{\sqrt{2 \pi t}}\int_{-\infty}^0 dy \eee^{-\frac{(x-y)^2}{2t}} \\\nonumber
&\times& \E\Biggl[\exp\Biggl(\int_0^t 
 \biggl(1-\tb-\o\left(x\sfrac {t-s}t +\sfrac st y +\zet^t_{0,0}(s)+a-\sqrt  2(t-s)\right)
\biggr)ds\Biggr)\Biggr].
\eea
To see how we can use these bounds, let us first ignore the possible excursions of the Brownian bridge and simply set $\zet_{0,0}^t(s)=0$.
We want to see where $w(t,x)$ drops from $1$ to zero. From \eqv(fk.7) we already know that this must happen before $x=\sqrt {2(1-\tb)} t$. Now assume that for some 
$u\leq \sqrt {2(1-\tb)}$, $w(t,ut +z)\leq \e$, for all $z\geq 0$. 
Then, for $z\geq 0$ independent of $t$, 
\bea\Eq(fk.10)
w(t,ut+z) &\geq& \nonumber
 \sfrac{(1-\tb/\tg)}{\sqrt{2 \pi t}}\int_{-\infty}^0dy \eee^{-\frac{(ut+z-y)^2}{2t}}
 \exp\Biggl(\int_0^t 
 \biggl(1-\tb-\o\left((u-\sqrt 2)(t-s) +z \sfrac {t-s}t+\sfrac st y +a\right)\nonumber\\
  &&\qquad\qquad-\tg w\biggl(t-s,
 u(t-s)+z\sfrac {t-s}t+\sfrac st y 
 \biggr)ds\Biggr)\nonumber\\
&\geq& \sfrac{(1-\tb/\tg)}{\sqrt{2 \pi t}}\int_{-\infty}^0dy \eee^{-\frac{(ut+z-y)^2}{2t}} \exp\left(t 
 (\tg-\tb- \e)\right)\nonumber\\
  &\geq&\sfrac{(1-\tb/\tg)}{\sqrt{2 \pi t}u} \eee^{-\frac {u^2t}2 - uz -z^2/(2t)+
  t (\tg-\tb- \e)}\sim  \eee^{-\frac{u^2t}2+  t (\tg-\tb- \e)},
 \eea
 which tends to infinity if $u<\sqrt {2(\tg-\tb)}$. Hence, the hypothesis can only be true for $u\geq \sqrt {2(\tg-\tb)}$.
 On the other hand, if $\sqrt 2 >u>\sqrt {2(\tg-\tb)}$, we get the corresponding 
 upper bound
\be\Eq(fk.11)
w(t,ut+z) \leq\sfrac{(1-\tb/\tg)}{\sqrt{2 \pi t}u} \eee^{-\frac {u^2t}2 - uz -z^2/(2t)+
  t (\tg-\tb)},
  \ee
  which is decaying exponentially with $t$. This  suggests a wave moving at speed 
  $u_0=\sqrt {2(\tg-\tb)}$, which is the speed we obtain if $v(0,x)\equiv 1$. 
  This shows that the only way to move faster is to exploit the possibility of the 
  Brownian bridge to make a forward excursion out of the region where $\o=1$.  
  \subsection{Improved heuristics on the wave speed}
  First, note that in \eqv(fk.5) $y$ is negative, so that we cannot gain anything from it and pretend that it is equal to zero in this subsection. To simplify the heuristics we also set $a=0$. Moreover, as we are analysing the possible gain in $\omega$ by large Brownian bridge excursions to areas where $\omega$ is small, we will ignore $w$ 
  (which is always way smaller than $\omega$) in \eqv(fk.5).
 Hence, we are left with estimating
 \be\Eq(heu.3)
 \frac{1}{\sqrt{2 \pi t}}  \eee^{-\frac{x^2}{2t}} 
\E\Biggl[\exp\Biggl(\int_0^t 
 \biggl(1-\tb-(1-\tg)\o\left(x\sfrac {t-s}t   +\zet^t_{0,0}(s)-\sqrt  2(t-s)\right)
\biggr)ds\Biggr)\Biggr].
 \ee
For our heuristics we approximate $\o$ by 
\be\Eq(heu.7)
\o\left(x\sfrac {t-s}t   +\zet^t_{0,0}(s)-\sqrt  2(t-s)\right)\approx \1_{x\sfrac {t-s}t   +\zet^t_{0,0}(s)-\sqrt  2(t-s)\leq 0}.
\ee
Hence, to further estimate the expectation in \eqv(heu.3) we need an estimate on the time during which the indicator function takes the value $0$. To this end, let
 \be\Eq(heu.8.1)
  T_t\equiv \int_0^t \1_{\zet_{0,0}^t(s)\geq \a(t- s)}ds,
 \ee
 with $\a=\sqrt{2}-x/t$, be the time the Brownian bridge spends above a line with slope $\a$. Note that \eqv(heu.3) is then approximately equal to
 \be\Eq(heu.100)
  \frac{1}{\sqrt{2 \pi t}}  \eee^{-\frac{x^2}{2t}} \eee^{\left(\tg-\tb\right)t}
\E\left[\eee^{(1-\tg)T_t}\right].
 \ee
 Next, on the exponential scale 
 \be
   \P(T_t>S)\approx  \P(T_t\approx S)\approx\P\left(\zet^t_{0,0}(S)\approx
    \left(\sqrt{2} -x/t\right)S\right) 
    =\sqrt{\sfrac{t}{2\pi (t-S)S}} \eee^{-\frac{\left(\sqrt{2} -x/t\right)^2St}{2(t-S)}},
 \ee
 where we used that heuristically the cheapest way to realise the event $\{T_t>S\}$ is to stay above this line up to roughly time $S$.  This probability is roughly dominated by the event to be essentially on the line at time $S$. As we gain a factor $(1-\tg)$ (on the exponential scale) as long as the Brownian bridge is above the line with slope $(\sqrt{2} -x/t)$, to find the dominating event in the expectation in \eqv(heu.3) we need to find the optimal $S^*$, namely
   \be\Eq(heu.1)
  S^*\equiv \mbox{argmax}_S \left(-S \sfrac {t\a^2}{2(t-S)} + (1-\tg)S\right).
  \ee
  By differentiating the right-hand side of \eqv(heu.1), we see that
  \be\Eq(heu.2)
  S^*= t\left(1-\sfrac{\sqrt{2}-x/t}{\sqrt{2(1-\tg)}}\right).
  \ee
 Now, we distinguish two cases. 
 \begin{itemize}
 \item[(Case 1)]If $S^*$ is positive, we 
 plug this back into \eqv(heu.100). Then the exponent in \eqv(heu.100) is to leading order equal to  
  \bea\Eq(heu.4)
&& -\sfrac{x^2}{2t}+ t(1-\tb)-t\sqrt{2(1-\tg)}\left(\sqrt{2}-x/t\right) +\sfrac{\left(\sqrt{2}-x/t\right)^2}{2}t\nonumber\\
 &=& -\tb t+2t(1-\sqrt{1-\tg})- \left(\sqrt{2}-\sqrt{2(1-\tg)}\right) x.
 \eea
 To see where $w$ starts to decay to $0$, we need to see for which $x$ \eqv(heu.4) is equal to zero (hence its exponential is of order $1$). This leads to
 \be\Eq(heu.5)
 x_1^* (\tb)=\sqrt{2}\left(1-\sfrac{\tb}{2\tg}\left(1+\sqrt{1-\tg}\right)\right)t.
\ee
\item[(Case 2)] If $S^*\leq 0$ in \eqv(heu.2), we cannot gain anything from the Brownian bridge excursion into areas where $\o\approx 0$ and always have $\o=1$. And thus the exponent in \eqv(heu.3) is approximately
\be\Eq(heu.6)
-\sfrac{x^2}{2t}+(\tg-\tb),
\ee
which is of order one for
\be\Eq(heu.8)
x_2^*(\tb)=\sqrt{2(\tg-\tb)}t.
\ee
\end{itemize}

We can summarise \eqv(heu.4) and \eqv(heu.6) by
\be\Eq(heu.7.1)
w(t,ut)\approx \exp\left({-t\left(\sfrac{u^2}{2}-\left(\tg-\tb\right)-\sfrac12{\left(\sqrt{2}-u-\sqrt{2\left(1-\tg\right)}\right)^2}\1_{u>\sqrt 2\left(1-\sqrt{1-\tg}\right)}\right)}\right).
\ee
The exponent is zero if $ut=x^*_2(\tb)$ 
and $u\leq\sqrt 2\left(1-\sqrt{1-\tg}\right)$ or if $ut=x^*_1(\tb)$  
and $u>\sqrt 2\left(1-\sqrt{1-\tg}\right)$.
Seeing $ x_1^*(\tb)$ as a function of $\tb$, we observe that it is decreasing in $\tb$ and there is exactly one critical value $\tb^*_1$ such that 
\be\Eq(heu.9)
 x_1^*(\tb^*_1)= \sqrt2\left(1-\sqrt{1-\tg}\right)t.
\ee
Namely,
\be\Eq(heu.10)
\tb_1^*=2\left(\tg+\sqrt{1-\tg}-1\right).
\ee
Similarly, seeing $ x_2^*(\tb)$ as a function if $\tb$ we observe that it is decreasing in $\tb$ and there is exactly one critical value $\tb^*_2$ such that 
\be\Eq(heu.9.1)
 x_2^*(\tb^*_2)= \sqrt 2\left(1-\sqrt{1-\tg}\right)t.
\ee
Namely,
\be\Eq(heu.10.1)
\tb_2^*=2\left(\tg+\sqrt{1-\tg}-1\right)=\tb_1^*.
\ee
As the two critical values for $\tb$ are the same, this suggests that for $\tb>\tb_1^*$ the speed of the wave equals $x_2^*/t$ and increases continuously to $x_1^*$ for $\tb<\tb_1^*$.
This will be made rigorous in the following sections.

\section{Brownian bridge estimates}

 In this section we provide 
  the key input about Brownian bridges that is needed to make the 
 heuristics above rigorous.
 \subsection{Probabilities of excursions}
  As $\o$ is not exactly an indicator function,
  the  key question is to know the distribution of the time a Brownian bridge $\zet_{0,0}^t$ 
  spends well above and well below a line $(\sqrt 2-u)(t-s)$, $0\leq s\leq t$. 
  Define, for $\a\equiv \sqrt 2-u$ fixed, for $K\in\R$, (see Figure \ref{lisa-bridge.2})
  \be
  \Eq(fk.12)
  T^{K}_t\equiv \int_0^t \1_{\zet_{0,0}^t(s)\geq \a(t- s)+ K}ds.
  \ee
  Note that 
  $\zet^t_{0,0}(s)$ has the same law as   $\zet^t_{0,0}(t-s)$, and so we can replace 
  $T_t^{K}$ by
  \be
  \Eq(fk.12.1)
  T^{K}_t= \int_0^t \1_{\zet_{0,0}^t(s)\geq \a s+ K}ds,
  \ee
  for convenience. The following theorem provides precise tail asymptotic for $T^{K}_t$.

 
 \begin{theorem}\TH(bb.1)
 Let $\zet_{0,0}^t$ be a Brownian bridge from zero to zero in time $t$. Let $\a>0$ and 
 $T_t$ defined in \eqv(fk.12). Then, for $0<s\leq 1$, 
 \bea \Eq(bb.2)
&& \P\left(T^K_t> st\right)
\\
&&= t^{-3/2}\a \sqrt{\sfrac 1{2\pi s^3(1-s)^3}} 
\eee^{-\frac{t\a^2s}{2(1-s)}-\frac{\a K}{1-s}}
\times
\begin{cases} 
\left(\sfrac{2(1-s)^2}{\a^2}\right)^2 {2} \left(1+o(1)\right),&\;\hbox{\rm if }\; K=0,
\nonumber\\
\left(\sfrac{2(1-s)^2}{\a^2}\right)^{3/2}
\sfrac{\sqrt{2}K}{\sqrt {\pi}} \left(1+o(1)\right),&\;\hbox{\rm if }\; K<0,
\end{cases}
\eea
and 
\be\Eq(bb2plus)
\P\left(T^K_t>st\right)
=t^{-3/2}
K(\sqrt \pi-1)
\sqrt{\sfrac { 1-s}{2\pi s^3}} 
\eee^{-\frac{t\a^2s}{2(1-s)}-\frac{\a K(1+\sqrt 2)}{1-s}}
\left(\sfrac{\sqrt 2K\a}{1-s}+1\right) (1+o(1)),
 \ee
if $K>0$.
 \end{theorem}

\begin{proof}
To start, 
we define $g_t$ as the last time the Brownian bridge $\zet_{0,0}^t$ is above the 
line $\a s+K$,  (see Figure \ref{lisa-bridge.2})
\be\Eq(g.1)
g_t\equiv \sup \left\{u\leq t: \zet_{0,0}^t (u)\geq \a u+K\right\}.
\ee
\begin{figure}\label{lisa-bridge.2}
\includegraphics[width=7cm]{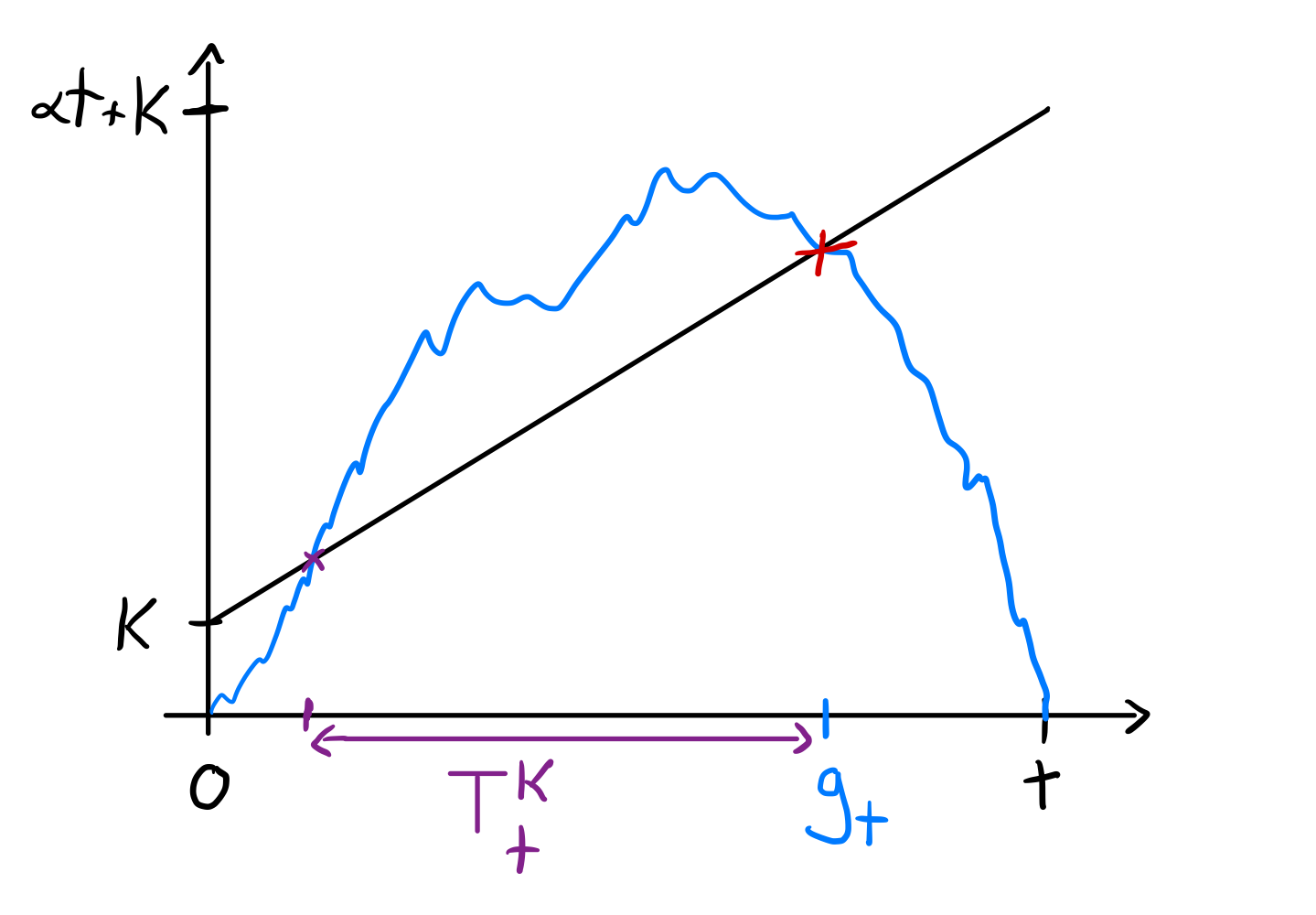}
\caption{Schematic picture of the Brownian bridge spending time $T^K_t$ above the line $\a s+K$.}
\end{figure}
Then 
\bea
\Eq(g.2)
\P\left(T^K_t>S\right)&=&\P\left(\int_0^{g_t} \1_{\zet_{-K,0}^{g_t}(u)\geq0}du\geq S\right)\nonumber\\
&=& \E\left[ \P\left(\int_0^{g_t} \1_{\zet_{-K,0}^{g_t}(u)\geq 0}du\geq S\big | g_t\right)\right]
\eea
 The conditional probability in \eqv(g.2) is known \cite{pech1,pech2}. A more convenient formula
is given in \cite{arzueda}, see Eq. (7) therein. For our setting this yields 
\bea
\Eq(pech.1)
&&\P\left(\int_0^{g_t} \1_{\zet_{-K,0}^{g_t}(u)\geq 0}du\geq S\big| g_t\right)\equiv 
\phi_K(g_t,S)
\nonumber\\
&&=
\begin{cases}-2\left(\sfrac{S}{g_t}\left(1-\sfrac{K^2}{g_t}\right)-1\right)\Phi\left(\sfrac{-K\sqrt{S}}{\sqrt{g_t(g_t-S)}}\right)
-\sfrac{K\sqrt{2S(g_t-S)}}{\sqrt{\pi g_t^3}}\eee^{-\frac {K^2S}{2g_t(g_t-S)}}, &\;\hbox{if}\; K\geq 0,\\
1+2\left(\sfrac{g_t-S}{g_t}\left(1-\sfrac{K^2}{g_t}\right)-1\right)\Phi\left(\sfrac{K\sqrt{g_t-S}}{\sqrt{g_tS}}\right)
-\sfrac{K\sqrt{2S(g_t-S)}}{\sqrt{\pi g_t^3}}\eee^{-\frac {K^2(g_t-S)}{2g_tS}}, &\;\hbox{if}\; K\leq 0,
\end{cases}
\eea
where $\Phi$ is the error function. Note that for $K=0$, this simplifies to 
\be
\Eq(pech.1.1)
\phi_0(g_t,S)=2\left(1-\sfrac S{g_t}\right)\Phi(0)=\left(1-\sfrac S{g_t}\right),
\ee
which recovers the result that the time spent by a Brownian bridge from $0$ to $0$ in time $g_t$ above $0$ is uniformly distributed on $[0,g_t]$.  

Next we need to control the distribution of $g_t$. Fortunately, this can  be recovered from known results by Beghin and 
Orsingher\cite{BegOrs}.

\begin{lemma}With the notation  above,
\TH(g.3)
\be\Eq(g.4)
\P\left(g_t\geq q\right)
=\eee^{-\frac{2K(\a t+K)}{t}}
\Phi\left(-\sfrac{(\a tq+K(2q-t))}{\sqrt{qt(t-q)}}\right) +1- \Phi\left(\sfrac{(\a tq)+Kt)} {\sqrt{qt(t-q)}}\right).
\ee
\end{lemma}

\begin{proof}
 Looking back in time, we see that we can also interpret $g_t$ as 
 \be g_t=t-\inf\left\{s>0: \zet_{0,0}^t(t-s)=\a (t-s)+K\right\}.
 \ee 
 \begin{figure}\label{lisa-bridge.1}
\includegraphics[width=7cm]{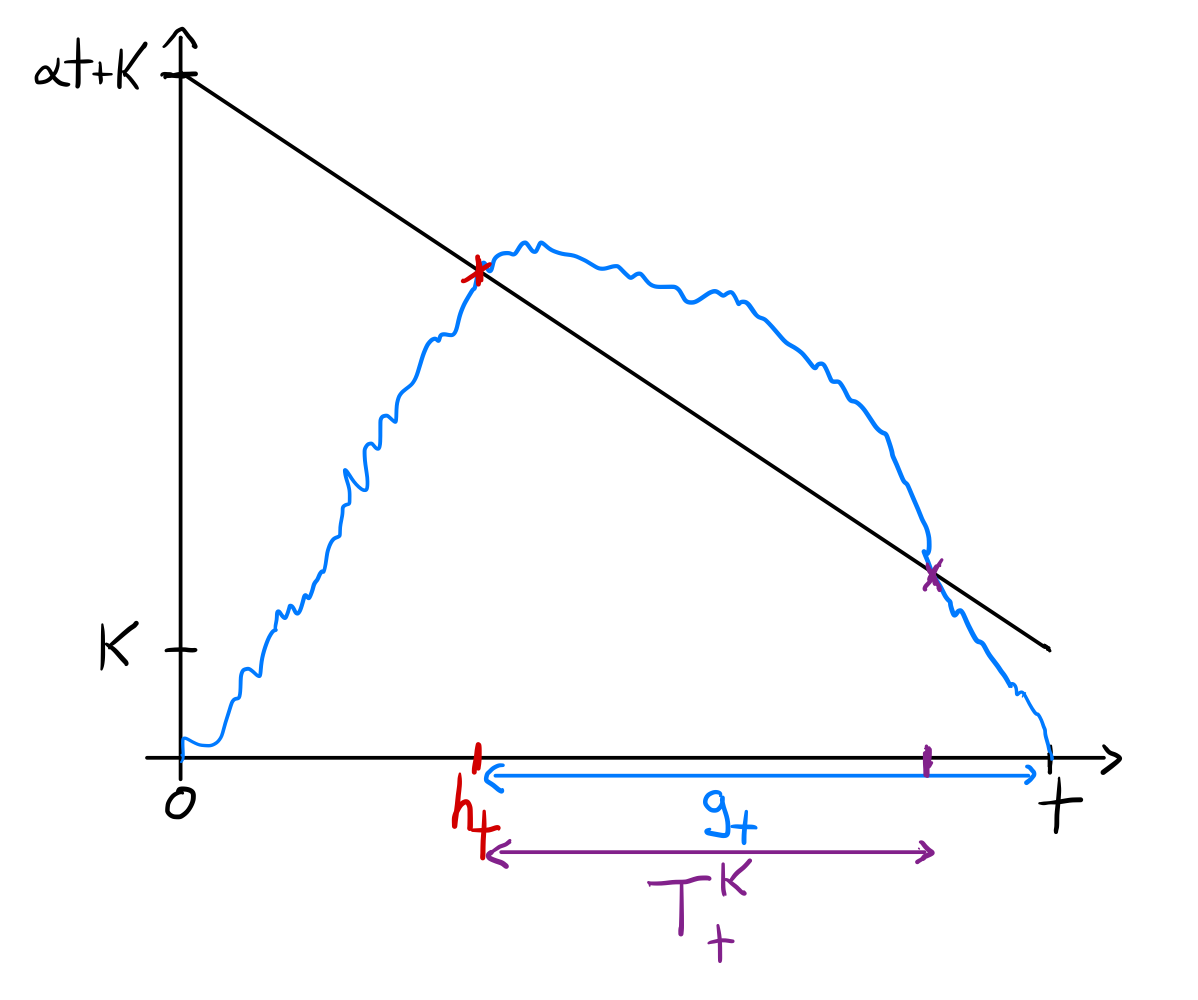}
\caption{Schematic picture of the Brownian bridge in reversed time}
\end{figure}

By time reversal,  this has the same law as $t-h_t$ where 
\bea
h_t&=&\inf\left\{ s>0: \zet_{0,0}^t(s) = \a(t-s)+K \right\}\nonumber\\ 
&=&\inf\left\{ s>0: \zet_{0,\a t}^t(s) = \a t+K \right\}.
\eea
The latter probability can be computed using a result by Beghin and Orsingher
\cite{BegOrs} (Lemma 2.1). It yields that, for $\a t+K>0$, 
\be\Eq(bo.1)
\P\left(h_t\leq r\right) = \eee^{-\frac{2K(\a t+K)}{t}}
\Phi\left(-\sfrac{(\a t(t-r)+K(t-2r))}{\sqrt{rt(t-r)}}\right)
+1- \Phi\left(\sfrac{(\a t(t-r)+Kt)} {\sqrt{rt(t-r)}}\right).
\ee
If $\a(t-r)+K\leq 0$, then this probability is equal to one.
Note that, in particular, 
\be\Eq(good.1)
\P\left(h_t\leq t\right)=\begin{cases} 
\eee^{-\frac{2K(\a t+K)}t},&\;\hbox{\rm if}\; K>0,\\
1,&\;\hbox{\rm if}\; K\leq0.
\end{cases}
\ee

Note that the term in the second line in \eqv(bo.1) is (asymptotically equal) and 
smaller than 
\be\Eq(funny.1)
\sqrt{\sfrac {r(t-r)}{2\pi(\a(t-r)+K)^2t}}  \eee^{-\frac{\a^2t(t-r)}{2r}-\frac{K\a t}{r}-\frac{K^2t}{2r(t-r)}}.
\ee
If $\a t(t-r)+K(t-2r)>0$, the first term in \eqv(bo.1) is asymptotically equal to and smaller than
\be\Eq(funny.2)
 \eee^{-\frac{2K(\a t+K)}{t}} 
 \sqrt{\sfrac {rt(t-r)}{2\pi(\a t(t-r)+K(t-2r))^2}}
 \eee^{-\frac{\a^2t(t-r)}{2r}-\frac{\a K(t-2r)}{r}-\frac {K^2(t-2r)^2}{2tr(t-r)}}.
 \ee

Recalling that $g_t=t-h_t$,  
we get that 
$\P\left(g_t\geq q\right)=\P\left(h_t\leq t-q\right)$ and hence the assertion of the lemma follows.\end{proof}

We  compute the probability density of the distribution of $g_t$
by differentiating \eqv(g.4). This gives the nice formula
\be\Eq(nice.1)
\P\left(g_t\in du\right)
=(\a t +K)\sqrt{\sfrac t{2\pi u(t-u)^3}} \eee^{-\frac{t(\a u+K)^2}{2u(t-u)}} du.
\ee
Thus, by \eqv(g.2),
\bea\Eq(notsonice.1)
\P\left(T^K_t>st\right)
&=& \int_{st}^t
\P\left(g_t\in du\right)\phi_K(u,st)\\ \nonumber
&=&\int_{st}^t (\a t +K)
\sqrt{\sfrac t{2\pi u(t-u)^3}} \eee^{-\frac{t(\a u+K)^2}{2u(t-u)}}
\phi_K(u,st)du\\\nonumber
&=&\int_{0}^{1-s} (\a  +K/t)
\sqrt{\sfrac 1{2\pi (s+v)(1-s-v)^3}} \eee^{-\frac{t(\a s+v+K)^2}{2(s+v)(1-s-v)}}
\sqrt t\phi_K(st+vt,st) dv,
\eea
where we used \eqv(pech.1) together with \eqv(nice.1). We  use the Laplace method to compute the integral in \eqv(notsonice.1). 
The exponential term takes its maximum  at $v=s$.
Thus we need
to compute the behaviour of the prefactor at $s$.
Let us first consider the more complicated case $K>0$.
We get
\bea\Eq(verynice.1)
\phi_K(st+x,st)
&=&-2\left(\sfrac{st}{st+x}\left(1-\sfrac{K^2}{st+x}\right)-1\right)\Phi\left(-\sfrac{K\sqrt{st}}{\sqrt{(st+x)x}}\right)
-\sfrac{K\sqrt{2st x}}{\sqrt{\pi (st+x)^3}}\eee^{-\frac {K^2st}{2(st+x)x}}\nonumber\\
&\sim&
\eee^{-\sfrac {K^2st}{2(st+x)x}}\sqrt{x(st+x)}\left(\sfrac{2x}{K\sqrt st (st+x)}+\sfrac{K }{st^{3/2}}
-\sfrac{K}{\sqrt{\pi(st+x)^3}}\right)\nonumber\\
&\sim&
\eee^{-\sfrac {K^2}{2x}}\sqrt{x}\sfrac{K }{st}(1-1/\sqrt \pi),
\eea
as $x\downarrow0$. Hence,
\be
t^{1/2}\phi_K(st+vt,st)=K\eee^{-\sfrac {K^2}{2vt}}\sfrac{\sqrt{v} }{s}(1-1/\sqrt \pi)(1+o(1)).
\ee
Similarly,
\be
\sqrt{\sfrac 1{2\pi (s+v)(1-s-v)^3}} \eee^{-\frac{t(\a s+v+K)^2}{2(s+v)(1-s-v)}}
\sim \sqrt{\sfrac 1{2\pi s(1-s)^3}} 
\eee^{-\frac{t\a^2s}{2(1-s)}-\frac{\a K}{1-s}}\eee^{-vt\frac{\a^2}{2(1-s)^2}}.
\ee
Inserting these asymptotics into \eqv(notsonice.1), 
we find that, up to errors of order $1/t$,
\be\Eq(quitenice.1)
\P\left(T^K_t>st\right)
=
\a K(1-1/\sqrt \pi)
\sqrt{\sfrac 1{2\pi s^3(1-s)^3}} 
\eee^{-\frac{t\a^2s}{2(1-s)}-\frac{\a K}{1-s}}
\int_{0}^{1-s} 
\eee^{-vt\frac{\a^2}{2(1-s)^2}- \frac{K^2}{2vt}}\sqrt v
 dv.
\ee
Finally, as $t\uparrow\infty$, substituting $z=vt\frac{\a^2}{2(1-s)^2}$,
\bea
\int_{0}^{1-s} 
\eee^{-vt\frac{\a^2}{2(1-s)^2}- \frac{K^2}{2vt}}\sqrt v
 dv
& \sim&\sfrac {2(1-s)^2}{\a t^{3/2}}\int_0^\infty 
\eee^{-z- \frac{K^2\a^2}{2z(1-s)^2}}\sqrt z
 dz\nonumber\\
 &=& \sfrac {(1-s)^2}{\a t^{3/2}}\sqrt \pi\left(\sfrac{\sqrt 2K\a}{1-s}+1\right)
 \eee^{-\frac{\sqrt 2 K\a}{1-s}}(1+o(1)),
\eea
so that finally 
\be\Eq(quitenice.2)
\P\left(T^K_t>st\right)
=t^{-3/2}
K(\sqrt{\pi}-1)
\sqrt{\sfrac { 1-s}{2\pi s^3}} 
\eee^{-\frac{t\a^2s}{2(1-s)}-\frac{\a K(1+\sqrt 2)}{1-s}}
\left(\sfrac{\sqrt 2K\a}{1-s}+1\right)(1+o(1)).
 \ee
In the remaining cases we get
\be\Eq(asympt.1)\phi_K(ts+tv,ts)
=\begin{cases}
\sfrac {v}{s}, &\;\hbox{\rm if }\; K=0,\\
t^{-1/2}\sfrac{\sqrt{2v}2|K|}{s\sqrt {\pi}}, &\;\hbox{\rm if }\; K<0.
\end{cases}
\ee
Therefore, using Lemma \thv(lem.laplace2),
\bea\Eq(lisa.lan1)
&&\int_{0}^{1-s} (\a  +K/t)
\sqrt{\sfrac 1{2\pi (s+v)(1-s-v)^3}} \eee^{-\frac{t(\a s+v+K)^2}{2(s+v)(1-s-v)}}
\sqrt t\phi_K(st+vt,st) dv
\\
&&=\left(1+o(1)\right)\a \sqrt{\sfrac 1{2\pi s(1-s)^3}} 
\eee^{-\frac{t\a^2s}{2(1-s)}-\frac{\a K}{1-s}}
\int_0^{1-s}dv
\eee^{-vt\frac{\a^2}{2(1-s)^2}}
\times
\begin{cases}
\sfrac {v}{s},&\;\hbox{\rm if }\; K=0 \\
t^{-1/2}\sfrac{\sqrt{2v}2|K|}{s\sqrt {\pi}}, &\;\hbox{\rm if }\; K<0\\
\end{cases}
\nonumber\\
&&= \left(1+o(1)\right)t^{-3/2}\a \sqrt{\sfrac 1{2\pi s^3(1-s)^3}} 
\eee^{-\frac{t\a^2s}{2(1-s)}-\frac{\a K}{1-s}}
\times
\begin{cases} 
\left(\sfrac{2(1-s)^2}{\a^2}\right)^2,&\;\hbox{\rm if }\; K=0,
\nonumber\\
\left(\sfrac{2(1-s)^2}{\a^2}\right)^{3/2}
\sfrac{\sqrt{2}|K|}{\sqrt {\pi}},&\;\hbox{\rm if }\; K<0.
\end{cases}
\eea
\eqv(lisa.lan1) and \eqv(quitenice.2) yield the assertion of Theorem \thv(bb.1).
\end{proof}
The control of the distribution of $T^K_t$ given by Theorem \thv(bb.1)  suffice to prove upper bounds on $w$ and hence upper bounds on the wave speed. To prove lower bounds, it is also necessary to take possible fluctuations of the Brownian bridges in the 
negative direction into account. Therefore, we need
on the distribution of $T^K_t$  a lower bound where 
excursions of the Brownian bridge below zero are suppressed. 
We define, for $b>0$, 
\be\Eq(low.1)
U_t^b\equiv \int_0^t \1_{\zet^t_{0,0}(s)\leq -b}ds.
\ee
We want a lower bound on 
\be\Eq(low.2)
\P\left(\{T^K_t>S\} \cap\{ U_t^b\leq L\}\right).
\ee
The following lemma is not optimal but sufficient for our purposes.
\begin{lemma}\TH(low.3) For $K> 0$, $b>0$, and $L>0$, 
 \be
 \Eq(low.4)
 \P\left(\{T^K_t>st\} \cap \{ U_t^b\leq L\}\right)\geq 
 C t^{-3/2}\eee^{-\sfrac {K^2}{2L}}\sqrt{L}K \sqrt{\sfrac {1-s}{2\pi \a^2}}
\eee^{-\frac{t\a^2s}{2(1-s)}-\frac{\a^2L+\a K}{1-s}}
\left(1-\eee^{-\frac{2b\a s}{1-s}}\right).
 \ee
 \end{lemma}
 \begin{proof}
  Given $g_t$, we  use that 
 \be
 \Eq(low.5)
 \left\{ U_t^b\leq L\right\}\supseteq \left\{ \zet^t_{0,0}(s) \geq -b, \forall g_t\leq s\leq t\right\}
 \cap \left\{ \int_0^{g_t} \1_{\zet_{0,0}^t(s)\leq -b} ds\leq L\right\}.
 \ee 
 The second event in turn contains the event that $\{S>g_t-L\}$.

Hence, the main effort is to control the law of $g_t$ under the restriction
that the bridge remains above $-b$. By the same reasoning as before,
this amounts to proving a lower bound on 
\be
\Eq(low.6)
\P\left(\{h_t\leq t-u \}\cap \{ \zet_{0,0}^t(s)\geq -b,\, \forall s\leq h_t\}\right).
\ee
To bound this, we have to revisit and alter the proof in \cite{BegOrs}.
First, we note that 
\bea\Eq(low.7)
&&\P\left(\{h_t\leq r\} \cap \{ \zet_{0,0}^t(s)\geq -b,\, \forall s\leq h_t\}\right)\nonumber\\
&&=\P\left(\left\{\max_{0\leq s\leq r} \zet_{0,0}^t(s)\geq \a(t-s) +K\right\} \cap
\left\{\min_{0\leq s\leq r}\zet_{0,0}^t(s)\geq -b\right\}\right).
\eea
The latter probability can be written up to normalisation as
\be
\Eq(low.8)
\P\left(\left\{\max_{0\leq s\leq r} B(s)\geq \a(t-s) +K\right\} \cap
\left\{\min_{0\leq s\leq r}B(s)\geq -b\right\}\cap \left\{B(t)=0\right\}\right),
\ee
where $B$ is a Brownian motion started in zero.
 Decomposing  this over the values of $B(r)$ gives
 \bea
\Eq(low.9)\nonumber
&&\P\left(\left\{\max_{0\leq s\leq r} B(s)\geq \a(t-s) +K\right\} \cap
\left\{\min_{0\leq s\leq r}B(s)\geq -b\right\}\cap \left\{B(t)=0\right\}\right)\\ \nonumber
&&=\int_{-b}^\infty 
\P\left(\left\{\max_{0\leq s\leq r} B(s)\geq \a(t-s) +K\right\} \cap
\left\{\min_{0\leq s\leq r}B(s)\geq -b\right\}\cap \left\{  B(r)\in dz\right\}\right)\\\nonumber
&&\qquad\times\P\left(B(t)=0\big | B(r)=z\right)\\\nonumber.
&&\geq\int_{\a(t-r)+K}^\infty 
\P\left(\left\{\max_{0\leq s\leq r} B(s)\geq \a(t-s) +K\right\} \cap
\left\{\min_{0\leq s\leq r}B(s)\geq -b\right\}\cap\left\{ B(r)\in dz\right\}\right)\\\nonumber
&&\qquad\times\P\left(B(t)=0\big | B(r)=z\right)\\
&&\equiv G^>(t-r).
\eea
Now, if  $z>\a(t-r)+K$ then $B(r)$ is above the line $\a(t-s)$ at $s=r$ and a fortiori
$\max_{0\leq s\leq r} B(s)\geq \a(t-s) +K$. Hence for these values of $z$,
\bea
\Eq(low.10)
&&\P\left(\left\{\max_{0\leq s\leq r} B(s)\geq \a(t-s) +K\right\} \cap
\left\{\min_{0\leq s\leq r}B(s)\geq -b\right\}\cap\left\{ B(r)\in dz\right\}\right)
 \nonumber\\&&=\P\left(\left\{
\min_{0\leq s\leq r}B(s)\geq -b\right\}\cap\left\{ B(r)\in dz\right\}\right)\nonumber\\
&&=\P\left(B(r)\in dz\right)-\P\left(\left\{
\min_{0\leq s\leq r}B(s)\leq -b\right\}\cap \left\{B(r)\in dz\right\}\right).
\eea
For the last probability we have by the reflection principle that
\be
\Eq(low.11)
\P\left(\left\{
\min_{0\leq s\leq r}B(s)\leq -b\right\}\cap\left\{B(r)\in dz\right\}\right)
=\P\left(B(r)\in d(-z-2b)\right).
\ee
The probability in \eqv(low.10) is thus given by
\be\Eq(low.12)
\frac 1{\sqrt{2\pi r}} \eee^{-\frac{ z^2}{2r}} \left(1-\eee^{-\frac{2bz+2b^2}r}\right) dz.
\ee
Hence, 
\bea
\Eq(low.13)
G^>(t-r)&=&
\frac 1{2\pi\sqrt{ r(t-r)}} 
\int_{\a(t-r)+K}^\infty \eee^{-\frac{ z^2}{2r}} \left(1-\eee^{-\frac{2bz+2b^2}r}\right) 
\eee^{-\frac{z^2}{2(t-r)}} dz\nonumber\\
&&=\frac 1{2\pi\sqrt{ r(t-r)}} 
\int_{\a(t-r)+K}^\infty \eee^{-\frac{ z^2t}{2r(t-r)}}  \left(1-\eee^{-\frac{2bz+2b^2}r}\right) dz.
\eea
Passing back to the Brownian bridge, this yields 
\be
\Eq(low.14)
\P\left(\{g_t\geq u\} \cap\{ \zet_{0,0}^t(s)\geq -b,\, \forall s\leq h_t\}\right)
\geq \sqrt{2\pi t}G^>(u)
\ee
Since $\phi_K(g_t, S)$ is monotone increasing in $g_t$, it holds that 
\bea
\Eq(low.15)
\P\left(\{T^K_t>S\} \cap\{ U_t^b\leq L\}\right)
&\geq& \int_{S+L}^t \P\left(\{g_t\in du\} \cap\left\{ \zet_{0,0}^t(s)\geq -b,\forall u\leq s\leq t\right\}\right)
\phi_K(u,u-L)\nonumber\\
&\geq& \min_{u\in [S+L,t]} \phi_K(u,u-L)
\sqrt{2\pi t} G^>(S+L).
\eea
Using \eqv(verynice.1), for $L$ finite and $S=st$,  
\bea
\Eq(low.16)
&&\P\left(\{T^K_t>st\} \cap  \{U_t^b\leq L\}\right)\nonumber\\
&&\sim\eee^{-\sfrac {K^2}{2L}}\sqrt{L}\sfrac{K }{t}(1-1/\sqrt \pi)
\sqrt{\sfrac { t}{{2\pi (st+L) (t-st-L)}} }
\int_{\a(st+L)+K}^\infty \eee^{-\frac{ z^2t}{2(st+L)(t-st-L)}}  \left(1-\eee^{-\frac{2bz+2b^2}{t-st-L}}\right) dz
\nonumber\\
&&
\geq \eee^{-\sfrac {K^2}{2L}}\sqrt{L}\sfrac{K }{t}(1-1/\sqrt \pi)
\sqrt{\sfrac{{(st+L)(t-st-L)}} {{2\pi (\a(st+L)+K)^2 t}}}
%
\eee^{-\frac{ t(\a( s+L/t)+K/t)^2}{2(s+L/t)(1-s-L/t)}}
\left(1-\eee^{-\frac{2b\a s}{1-s}}\right)
\nonumber\\
&&
\geq \eee^{-\sfrac {K^2}{2L}}\sqrt{L}\sfrac{K }{t}(1-1/\sqrt \pi)
\sqrt{\sfrac{(1-s)}{2\pi \a^2 st }}
\eee^{-\frac{t\a^2s}{2(1-s)}-\frac{\a^2L+\a K}{1-s}}
\left(1-\eee^{-\frac{2b\a s}{1-s}}\right)
\left(1+O(1/t)\right)
\nonumber\\
&&
\geq C t^{-3/2}\eee^{-\sfrac {K^2}{2L}}\sqrt{L}K \sqrt{\sfrac {1-s}{2\pi \a^2}}
\eee^{-\frac{t\a^2s}{2(1-s)}-\frac{\a^2L+\a K}{1-s}}
\left(1-\eee^{-\frac{2b\a s}{1-s}}\right),
\eea
for some $C>0$.
 \end{proof}
 \begin{remark}
 Note that, up to constants, the difference between the expression for $P(T^K_t>st)$ is 
that a factor $1/s$ is missing; this is due to the lower bound in \eqv(low.15).
 To keep the difference in upper and lower bound of polynomial order in $t$ one could choose
 $L\sim K$ and 
 \be\label{L.1}
 |K|\leq C\log(t).
 \ee
 \end{remark}
\subsection {The Laplace transforms}

As seen in \eqv(heu.100), we need to control the Laplace transform of $T^K_t$.  The behaviour of the Laplace transform is very different weither $2\l>\a^2$ or $2\l\leq \a^2$. 
\begin{lemma}\TH(bb.30)
Assume that $2\l>\a^2$. Then, as  $t\uparrow \infty$, 
\bea
\Eq(bb.31)
\E\left[\eee^{\l  T^K_t}\right] 
=
\eee^{t\frac{(\a-\sqrt{2\l})^2}2-K\sqrt{2\l}}
\sfrac {\sqrt{2\a}}{\sqrt{\pi (\sqrt{2\l}-\a)^3}}\times
\begin{cases} 
  \left(1+o(1)\right),&\;\hbox{\rm if }\; K=0,
\nonumber\\
{\sqrt{2\l}|K|} \left(1+o(1)\right),&\;\hbox{\rm if }\; K<0,
\end{cases}
\eea
and, if $K>0$, 
\be
\E\left[\eee^{\l  T^K_t}\right] 
=
\eee^{t\frac{(\a-\sqrt{2\l})^2}2-K\sqrt{2\l}(1+\sqrt 2)}
K(\sqrt{\pi}-1)\sqrt{\sfrac{\a\l}{4\pi(\sqrt{2\l}-\a)^3}}\left(2K\sqrt\l+1\right)(1+o(1)).
\ee
\end{lemma}

\begin{proof} Note first that for any non-negative random variable $T$, 
\be
\Eq(bb.32)
\E\left[\eee^{\l T}\right] =\int_0^\infty \P\left(T\geq \sfrac {\ln y}{\l}\right)dy
=\int_{-\infty}^\infty  \l \eee^{\l r}\P\left(T\geq r\right)dr.
\ee
From Theorem \thv(bb.1) we see that, for $0<r<t$, 
\be
\P\left(T^K_t\geq r\right) =t^{-3/2}P_K(t,r/t) \eee^{-\frac{\a^2 tr}{2(t-r)}},
\ee
where 
$ P_K(t,s)$  is polynomially bounded. Moreover, 
$\P\left(T\geq r\right) =1$ for  all $r\leq 0$ and $\P\left(T\geq t\right)=0$.Thus 
\bea
\Eq(bb.34)
\E\left[\eee^{\l T_t}\right] &=& \int_{-\infty}^0 \l \eee^{\l r}dr +t^{-3/2} \int_{0}^t \l \eee^{\l r} 
 \eee^{-\frac{\a^2t r}{2(t-r)}} P_K(t,r/t)(1+o(1))dr
 \nonumber\\
 &=&1+  t^{-1/2}\int_{0}^1 \l \eee^{\l st} 
 \eee^{-t\frac{\a^2 s}{2(1-s)}} P_K(t,s) (1+o(1))ds
 \eea
Let
 \be
 \Eq(bb.35)
 f(s)=\l s-\sfrac{\a^2s}{2(1-s)}, \quad s\in (0,1).
 \ee
$f(s)$ takes its maximum  in $(0,1)$   
at
 \be\Eq(bb.36)
s^*=1-\sfrac{\a}{\sqrt{2\l}},
 \ee
 provided that   $2\l>\a^2$. 
 By an elementary computation, 
 \be
 f(s^*)=\sfrac 12\left(\a-\sqrt{2\l}\right)^2,
 \ee
 and
  the second derivative of $f$ at $s^*$ is given by 
 \be
 \Eq(bb.37)
 f''(s^*) =-\sfrac{2\l\sqrt{2\l}}{\a}<0.
 \ee
 Using Lemma \thv(lem.laplace) with $f(s)$ as in \eqv(bb.35)
 we get that, if $2\l>\a^2$,  
\be \Eq(bb.38)
\E\left[\eee^{\l T_t^K}\right] =1+\l \sfrac{\sqrt{\a}}{(2\l)^{3/2}} \eee^{\frac {t\left(\a-\sqrt {2\l}\right)^2}{2}}
P_K(t,s^*).
\ee
where
\be
P_K(t,s^*)=\sfrac {2\l}{\sqrt{2\pi \left(1-\a/\sqrt{2\l}\right)^3}} \eee^{
-{ K\sqrt{2\l}}}
\times
\begin{cases} 
\left(\sfrac 1{\l}\right)^2 {2} \left(1+o(1)\right),&\;\hbox{\rm if }\; K=0,
\\
\left(\sfrac1{\l}\right)^{3/2}
\sqrt{2}|K| \left(1+o(1)\right),&\;\hbox{\rm if }\; K<0,
\end{cases}
\ee
and 
\be
P_K(t,s^*)=
K\left(\sqrt \pi-1\right)
\sfrac {2\l}{\sqrt{2\pi\left(\sqrt{2\l}-\a\right)^3}} 
\eee^{-\sqrt{2\l} K(1+\sqrt 2)}
\left({\sqrt 2K\sqrt{2\l}}+1\right) (1+o(1)),
\ee
if $K>0$.
The claim of Lemma \thv(bb.30) follows.
\end{proof}
For the lower bound on $w$, we need to take negative excursions into account, as already mentioned in Section 4.1. The following Lemma provides a corresponding bound on the Laplace transform.
\begin{lemma}\TH(lem.lap.low)
Assume that $2\lambda>\alpha^2$ and let $K>0$. Then, as $t\uparrow\infty$, 
\be\Eq(L.2)
\E\left[\eee^{\l T_t^K}\1_{U_t^b\leq L}\right]
\geq 
C \eee^{-\sfrac {K^2}{2L}}\sqrt{L}K \sqrt{\sfrac {1}{\left(\sqrt{2\lambda}-\a\right)2\pi \a}}\eee^{t\frac{(\a-\sqrt{2\l})^2}2-K\sqrt{2\l}(1+\sqrt 2)}\left(1-\eee^{-2b\left(\sqrt{2\l}-\a\right)}\right).
\ee
\end{lemma}
\begin{proof}
The proof is a rerun of the  proof of Lemma \thv(bb.30) using Lemma \thv(low.3) instead of Theorem \thv(bb.1).
\end{proof}
Finally, we need  bounds on the Laplace transform when $2\l\leq\a^2$. The following lemma confirms that in this case the 
Laplace transform is essentially of order one.
\begin{lemma}\TH(lem.small)
If $2\l\leq\a^2$ and $K<0$,
\be\Eq(lisa.lan5)
1\leq\E\left[\eee^{\l  T^K_t}\right] \leq 1+ \begin{cases}
  \frac{2}{\sqrt{\pi K}} \l \eee^{-2\l K/\a}\frac{2}{\a^2-2\l}, & \mbox{if }2\l<\a^2\\
    \frac{2}{\sqrt{\pi K}} \l \eee^{-2\l K/\a} (t+2K/\a), & \mbox{if }2\l=\a^2
 \end{cases}.
\ee
and 
\be\Eq(lisa.lan10)
\E\left[\eee^{\l T_t^K}\1_{U_t^b\leq L}\right] \geq 1-\eee^{-b^2/2t}.
\ee
\end{lemma}
\begin{proof}
Starting from \eqv(bb.32) we simply use
\be\Eq(lisa.lan6)
\P(T_t^K\geq r)\leq \P(g_t\geq r).
\ee
The latter has been computed in Lemma \thv(g.3) and we have, for $\a r>-2K$,
\bea\Eq(lisa.lan7)
 \P(g_t\geq r)&=&\eee^{-\frac{2K(\a t+K)}{t}}
\Phi\left(-\sfrac{(\a tr+K(2r-t))}{\sqrt{rt(t-r)}}\right) +1- \Phi\left(\sfrac{(\a tr)+Kt)} {\sqrt{rt(t-r)}}\right)\\
&\leq &
\eee^{-2K\a} \frac{\sqrt{rt(t-r)}}{(\a tr+K(2r-t))\sqrt{2\pi}}\eee^{-\frac{\left(\a tr+K(2r-t)\right)^2}{2rt(t-r)}} \nonumber
+\frac{\sqrt{rt(t-r)}}{(\a tr+Kt)\sqrt{2\pi}}\eee^{-\frac{\left(\a tr+Kt\right)^2}{2rt(t-r)}}\nonumber\\
&\leq&\frac{\sqrt{rt(t-r)}}{\sqrt{2\pi}} \eee^{-\frac{\a K t}{t-r} -\frac{\a^2 tr}{2(t-r)}}\left(\frac{1}{\a tr+K(2r-t)}+ \frac{1}{\a tr+Kt}\right)\nonumber\\
&\leq&\frac{4}{\sqrt{2\pi\a r}} \eee^{-\frac{\a K t}{t-r}-\frac{\a^2 tr}{2(t-r)}},
\eea
by standard Gaussian tail bounds. Plugging this into \eqv(bb.32) we get
\bea\Eq(lisa.lan8)
\E\left[\eee^{\l  T^K_t}\right]& \leq& \int_{-\infty}^{-2K/\a} \l \eee^{\l r} dr + \int_{-2K/\a}^t  \l \eee^{\l r}\frac{4}{\sqrt{2\pi\a r}} \eee^{-\frac{\a K t}{t-r}-\frac{\a^2 tr}{2(t-r)}}dr\nonumber\\
&=& e^{-2K\l/\a} + \int_{0}^{t+2K/\a}  \l \eee^{\l z-\l 2K/\a}\frac{4}{\sqrt{2\pi\a (z-2K/\a)}} \eee^{ -\frac{\a^2 tz}{2(t-z+2K/\a)}}dz,
\eea
where $z=r+2K/\a$.
The second summand in \eqv(lisa.lan8) is bounded from above by
\be\Eq(lisa.lan9)
 \frac{2}{\sqrt{\pi K}} \l \eee^{-2\l K/\a} \int_{0}^{t+2K/\a}  \eee^{\l z -\frac{\a^2 z}{2}}dz  =
 \begin{cases}
  \frac{2}{\sqrt{\pi K}} \l \eee^{-2\l K/\a}\frac{2}{\a^2-2\l}, & \mbox{if }2\l<\a^2\\
    \frac{2}{\sqrt{\pi K}} \l \eee^{-2\l K/\a} (t+2K/\a), & \mbox{if }2\l=\a^2
 \end{cases}.
\ee
To prove \eqv(lisa.lan10) we bound the left hand side of \eqv(lisa.lan10) from below by
\be\Eq(lisa.lan11)
\P\left(U_t^b\leq L\right)\geq \P\left(U_t^b=0\right) = 1-\eee^{-b^2/2t}.
\ee
This finishes the proof of Lemma \thv(lem.small).
\end{proof}

\section{Controlling the wave}
 
 We use the Brownian bridge estimates from the previous section to give a rigorous version of the heuristics outlined at the end of Section 3.

 \subsection{Bounds on the speed of the wave}
 We first control the behaviour of solutions on the exponential scale for large $t$.
  We begin with an upper bound.
 
 \begin{lemma}
 \TH(speed.10) \begin{itemize}
 \item[(i)] Let $u$ be such that
 \be\Eq(lisa.lan3)
 2\left(1-\tg\right)> \left(\sqrt{2}-u\right)^2,
 \ee
Then, for all $\e>0$ small enough, there exists a constant $C>0$ such that 
 \be\Eq(speed.2.2)
w(t,ut)\leq \frac{C}{u\sqrt{t}}\eee^{t \left(2-\tb-2\sqrt{(1-\tg)(1-\e)}+\sqrt 2 u\left(\sqrt{(1-\tg)(1-\e)}-1\right)\right) }.
 \ee
In particular, $w(t,ut)$ decays exponentially fast in $t$ for
 \be\Eq(speed.3)
 u>u^*\equiv \sqrt 2-\frac{\tb\left(\sqrt{1-\tg}+1\right)}{\sqrt 2 \tg}.
 \ee
 \item[(ii)] Let $u$ be such that
 \be\Eq(lisa.lan4)
 2\left(1-\tg\right)\leq \left(\sqrt{2}-u\right)^2,
 \ee
Then, for all $\e>0$ small enough, there exists a constant $C>0$ such that 
 \be\Eq(speed.2.1)
w(t,ut)\leq  \frac{C}{u\sqrt t}\eee^{-\sfrac{u^2t}{2}+t( (\tg-\tb) +\e(1-\tg)) }
 \ee
 \end{itemize}
 \end{lemma}
 \begin{remark} 
 Lemma \thv(speed.10) implies that the  solution is exponentially small if $u>u_c$ (given in \eqv(lisa.lan23)), hence the wave speed is not larger than $u_c$.
  \end{remark}
 
 \begin{proof}
 We bound the integral in the Feynman-Kac representation \eqv(fk.5) from above as follows.
 \bea
 \Eq(speed.14)
 &&\int_0^t 
 \biggl(1-\tb-(1-\tg)\o\left(x\sfrac {t-s}t +\sfrac st y +\zet^t_{0,0}(s)+a-\sqrt  2(t-s)\right)\nonumber\\
 &&\quad-\tg w\left(t-s, x\sfrac {t-s}t +\sfrac st y
 +\zet_{0,0}^t(s)\right)
\biggr)ds\nonumber\\
&&\leq
\int_0^t \left(\1_{\zet_{0,0}^t(s)\geq (\sqrt  2-x/t)(t-s) +K}
+\1_{\zet_{0,0}^t(s)< (\sqrt  2-x/t)(t-s) +K}\right)\nonumber\\
&&\quad\times
 \left(1-\tb-(1-\tg)\o\left(x\sfrac {t-s}t +\sfrac st y +\zet^t_{0,0}(s)+a-\sqrt  2(t-s)\right)\right)ds.
\eea
Using the asymptotic of the lower tail \eqv(fkpp.15),
we see that on the second indicator function, for all $y\leq 0$,
\be\Eq(speed.15)
 \o\left(x\sfrac {t-s}t +\sfrac st y +\zet^t_{0,0}(s)+a-\sqrt  2(t-s)\right)
 \geq 
 \o\left(a+K \right)\geq 1-c \eee^{(2-\sqrt 2)(a+K)},
 \ee
which is larger than $1-\e$ for $-K$ large enough. On the first indicator function we use 
that $\o\geq 0$. This leads to 
\bea\Eq(speed.16)
&&\int_0^t 
 1-\tb-(1-\tg)\o\left(x\sfrac {t-s}t +\sfrac st y +\zet^t_{0,0}(s)+a-\sqrt  2(t-s)\right)
ds\nonumber\\
&&\leq(\tg-\tb)t+\e(1-\tg)t+(1-\tg)(1-\e)T^{K}_t.
\eea
 Recalling \eqv(fk.8), if $ 2\left(1-\tg\right)> \left(\sqrt{2}-u\right)^2$, we obtain the upper bound
\bea\Eq(speed.17)\nonumber
 w(t,x)&\leq& \sfrac {1-\tb/\tg}{\sqrt{2\pi t}}\int_{-\infty}^0 \eee^{-\sfrac{(x-y)^2}{2t}}
 \eee^{t( (\tg-\tb) +\e(1-\tg) )} \E\left[\eee^{(1-\tg)(1-\e)T^K_t}\right]dy
\\ \nonumber
&\leq& \sfrac {1-\tb/\tg}{\sqrt{2\pi x^2/t}} \eee^{-\sfrac{x^2}{2t}}
 \eee^{t( (\tg-\tb) +\e(1-\tg)) }
 \eee^{t{\left((\sqrt 2-x/t)-\sqrt{2(1-\tg)(1-\e)}\right)^2}/2-K\sqrt{2(1-\tg)(1-\e)}}
\nonumber\\&&\times\sfrac {\sqrt{2(\sqrt 2-x/t)}}{\sqrt{\pi (\sqrt{2(1-\tg)(1-\e)}-(\sqrt 2-x/t))^3}}
{\sqrt{2(1-\tg)}|K|} \left(1+o(1)\right) .
 \eea
The exponential terms are  (for $x=ut$),
\be
\eee^{-t\left(\sfrac{u^2}{2}+
  \tb-\tg -\e(1-\tg)
-{\left((\sqrt 2-u)-\sqrt{2(1-\tg)(1-\e)}\right)^2}/2\right)}
 = \eee^{t \left(2-\tb-2\sqrt{(1-\tg)(1-\e)}+\sqrt 2 u\left(\sqrt{(1-\tg)(1-\e)}-1\right)\right)}.
\ee
This implies the first part of the Lemma.

For the second one, if  $2\left(1-\tg\right)\leq\left(\sqrt{2}-u\right)^2$ we use the bound from Lemma \thv(lem.small) and get
\bea\Eq(speed.17.1)
 w(t,x)&\leq& \sfrac {1-\tb/\tg}{\sqrt{2\pi t}}\int_{-\infty}^0 \eee^{-\sfrac{(x-y)^2}{2t}}
 \eee^{t( (\tg-\tb) +\e(1-\tg) )} \E\left[\eee^{(1-\tg)(1-\e)T^K_t}\right]dy
\\ \nonumber
&\leq& \sfrac {1-\tb/\tg}{\sqrt{2\pi x^2/t}} \eee^{-\sfrac{x^2}{2t}}
 \eee^{t( (\tg-\tb) +\e(1-\tg)) }\nonumber\\&&\nonumber
\times \begin{cases}
1+  \frac{2(1-\tg)}{\sqrt{\pi K}} \eee^{-2(1-\tg)(1-\e) K/\a}\frac{2}{\a^2-2(1-\tg)(1-\e)}, & \mbox{if }2(1-\tg)(1-\e)<\a^2,\\
 1+   \frac{2(1-\tg)}{\sqrt{\pi K}}  \eee^{-2(1-\tg)(1-\e) K/\a} (t+2K/\a), & \mbox{if }2(1-\tg)(1-\e)=\a^2.
 \end{cases}
 \eea
 The exponential terms are,  for $x=ut$,
 \be\Eq(speed.100)
\eee^{-\sfrac{u^2t}{2}+t( (\tg-\tb) +\e(1-\tg)) }
 \ee
This implies that $w$ decays exponentially fast for $u>\sqrt{2(\tg-\tb)}$.
 \end{proof}

Next, we need a corresponding lower bound. For this we use the lower bound from Lemma \thv(lem.lap.low).
 \begin{lemma}
 \TH(speed.1) Let $b>0$.
 Assume that $u$ is such that $w(t,ut+z)\leq \e$, for all $z\geq-2b$ and $t$ large enough.
 \begin{itemize}
 \item[(i)] 
If $ 2\left(1-\tg\right)\left(1-\e\right)> \left(\sqrt{2}-u\right)^2$
 then, for some constant $C>0$ depending on $K,L,b$ and $u$,
 \be\Eq(speed.2)
 w(t,ut) \geq {C} t^{-1/2}\eee^{-tu\left(\sqrt{2}-\sqrt{2(1-\tg)(1- \e)}\right)}\eee^{t\left(1-\sqrt{(1-\tg)(1- \e)}\right)^2t} \eee^{(\tg(1-\e)-\tb) t} .
 \ee
 This contradicts the hypothesis, unless
 \be\Eq(speed.3.1)
 u> \sqrt 2-\sfrac{\left(\tb+\e\right)\left(\sqrt{\left(1-\tg\right)\left(1-\e\right)}+1\right)}{\sqrt 2\left( \tg(1-\e) +\e\right)}.
 \ee
 \item[(ii)] If $ 2\left(1-\tg\right)\left(1-\e\right)\leq \left(\sqrt{2}-u\right)^2$, then there exists a constant $C>0$ depending on $K,b$ and $u$
 \be\Eq(lisa.lan14)
 w(t,ut)\geq  \frac{C}{t^{3/2}} \eee^{(\tg(1-\e)-\tb) t-\frac{u^2t}{2}} \ee
  This contradicts the hypothesis unless
\be\Eq(lisa.lan15)
u>\sqrt{2\left(\tg(1-\e)-\tb\right)}.
\ee
  \end{itemize}

 \end{lemma}
 \begin{remark}
From \eqv(speed.3.1) it follows that the speed is not smaller than $\sqrt{2}-\frac{\tb}{\sqrt{2}\tg} \left(1+\sqrt{1-\tg}\right)$ and from \eqv(lisa.lan15)
 it follows that it is not smaller than  $\sqrt{2\left(\tg-\tb\right)}$. Altogether, this implies 
 that the speed is not smaller than the maximum of the two, i.e. it is at least
 \end{remark}
 
 \begin{proof}
In the representation \eqv(fk.5) of $w(t,x)$ we would like to use the assumption of the 
lemma to argue that the term involving $w$ in the exponent is negligible but this could 
be spoiled by large negative excursions of the Brownian bridge. To avoid this problem 
we restrict the expectation in \eqv(fk.5) on the Brownian bridge to a subset 
$\{U_t^b\leq L\}$, For any $L>0$ and $b>0$, 
 \bea\Eq(L.3)
w(t,x)\geq&& \sfrac{1}{\sqrt{2 \pi t}}\int_{-b}^0dy \eee^{-\frac{(x-y)^2}{2t}} \nonumber\\
&\times&  \E\Biggl[\exp\Biggl(\int_0^t 
 \biggl(1-\tb-(1-\tg)\o\left(x\sfrac {t-s}t +\sfrac st y +\zet^t_{0,0}(s)+a-\sqrt  2(t-s)\right)\nonumber\\
 &&-\tg w\left(t-s, x\sfrac {t-s}t +\sfrac st y
 +\zet_{0,0}^t(s)\right)
\biggr)ds\Biggr)\1_{U_t^b\leq L}\Biggr],
\eea
Note that on the event $\{U_t^b\leq L\}$, we have
\be\Eq(L.4)
-\int_0^t \tg w\left(t-s, x\sfrac {t-s}t +\sfrac st y
 +\zet_{0,0}^t(s)\right)ds
\geq
-\tg L-\tg\e t.
\ee
Hence, \eqv(L.3) is bounded from below by
 \bea\Eq(L.3.1)
&& \eee^{-\tg L+(1-\tb-\tg\e) t}\sfrac{1-\tb/\tg}{\sqrt{2 \pi t}}\int_{-b}^0dy \eee^{-\frac{(x-y)^2}{2t}} \\ \nonumber
&\times& \E\Biggl[\exp\Biggl(-\int_0^t 
 (1-\tg)\o\left(x\sfrac {t-s}t +\sfrac st y +\zet^t_{0,0}(s)+a-\sqrt  2(t-s)\right)ds\Biggr)\1_{U_t^b\leq L}\Biggr].
 \eea
 The idea is to split the 
 integral in the exponent of the Feynman-Kac formula \eqv(fk.5)
 according to the position of the Brownian bridge 
 with respect to the $\o$-wave, i.e. we write, with $x=ut+z$ and $u$ as in the lemma,
 \bea
 \Eq(speed.4)\nonumber
 &&-\int_0^t 
 (1-\tg)\o\left(x\sfrac {t-s}t +\sfrac st y +\zet^t_{0,0}(s)+a-\sqrt  2(t-s)\right)
ds\nonumber\\
&&=
-\int_0^t \left(\1_{\zet_{0,0}^t(s)\geq (\sqrt  2-x/t)(t-s) +K}
+\1_{\zet_{0,0}^t(s)< (\sqrt  2-x/t)(t-s) +K}\right)
\nonumber\\
&&\quad\times
(1-\tg)\o\left(x\sfrac {t-s}t +\sfrac st y +\zet^t_{0,0}(s)+a-\sqrt  2(t-s)\right)ds.
\eea
 On the first indicator function we use that 
 \be\Eq(speed.5)
 \o\left(x\sfrac {t-s}t +\sfrac st y +\zet^t_{0,0}(s)+a-\sqrt  2(t-s)\right)
 \leq 
 \o\left(K+\sfrac st y +a \right)\leq C (K +a)\eee^{-\sqrt 2(K+a)/2},
 \ee
if $y\geq - (K+a)/2$. Choosing $K+a$ large enough, 
we can make this smaller than $\e$, for any $\e>0$. 
 On the second indicator function, we just use that $\o\leq 1$.
Thus \eqv(speed.4) is bounded from below by
\be
\Eq(speed.6)
-(1-\tg)\e T_t^K-(1-\tg)(t-T_t^K)
=-(1-\tg)t-(1-\tg)(1-\e)T_t^K.
\ee
 Inserting this bound into \eqv(fk.5), we get that
 \be
 \Eq(speed.7)
 w(t,x)\geq  \eee^{-\tg L+(\tg(1-\e)-\tb) t}\sfrac {1-\tb/\tg}{\sqrt{2\pi t}}\int_{-\left((K+a)/2 \land b\right)}^0 \eee^{-\frac{(x-y)^2}{2t}}
 \E\left[\eee^{(1-\tg)(1- \e)T^K_t}\1_{U_t^b\leq L}\right]dy.
 \ee
{\bf Case 1:} $2\left(1-\tg\right)\left(1-\e\right)> \left(\sqrt{2}-u\right)^2$.

 We insert the lower bound from Lemma \thv(lem.lap.low) into \eqv(speed.7).
 This gives, if $K>0$, 
 \bea
 \Eq(speed.71)
 w(t,x)&\geq& 
 \eee^{-\tg L+(\tg(1-\e)-\tb) t}\sfrac {1-\tb/\tg}{\sqrt{2\pi t}}\int_{-\left((K+a)/2 \land b\right)}^0 \eee^{-\frac{(x-y)^2}{2t}}dy C \eee^{-\sfrac {K^2}{2L}}\sqrt{L}K \sqrt{\sfrac {1}{\left(\sqrt{2(1-\tg)(1- \e)}-\a\right)2\pi \a}}\nonumber\\\nonumber
 &&\times \eee^{t\frac{\left(\a-\sqrt{2(1-\tg)(1- \e)}\right)^2}2-K\sqrt{2(1-\tg)(1- \e)}(1+\sqrt 2)}\left(1-\eee^{-2b\left(\sqrt{2(1-\tg)(1- \e)}-\a\right)}\right)\\\nonumber
& \geq& \eee^{-\frac{x^2}{2t}}\frac{1}{u}\left(1-\eee^{-\left((K+a)/2 \land b\right) u }\right)
 \eee^{-\tg L+(\tg(1-\e)-\tb) t}\sfrac {1-\tb/\tg}{\sqrt{2\pi t}} C \eee^{-\sfrac {K^2}{2L}}\sqrt{L}K \sqrt{\sfrac {1}{\left(\sqrt{2(1-\tg)(1- \e)}-\a\right)2\pi \a}}\nonumber\\
  &&\times\eee^{t\frac{\left(\a-\sqrt{2(1-\tg)(1- \e)}\right)^2}2-K\sqrt{2(1-\tg)(1- \e)}(1+\sqrt 2)}\left(1-\eee^{-2b\left(\sqrt{2(1-\tg)(1- \e)}-\a\right)}\right),
 \eea
 for $x=ut$.
%
 The exponential terms are, for $x=ut$,
 \bea
&& \eee^{-t \left(\frac {u^2}{2}  -\tg(1-\e) +\tb-\frac{(\sqrt 2-u-\sqrt{2(1-\tg)(1-\e)})^2}2\right)}
= \eee^{t \left(2-\e-\tb-2\sqrt{(1-\tg)(1-\e)}+\sqrt 2 u\left(\sqrt{(1-\tg)(1-\e)}-1\right)\right)}.
 \eea
 The exponent vanishes if 
 \be\Eq(star.1)
 u=u^*\equiv 
 \sqrt 2-\sfrac {(\tb+\e)\left(\sqrt {(1-\tg)(1-\e)}+1\right)}{\sqrt 2 (\g +\e(1-\tg))}
 =\sqrt 2-\sfrac {\tb\left(\sqrt {(1-\tg)}+1\right)}{\sqrt 2 \g } +O(\e).
 \ee
 and is decreasing in $u$. Hence, for $u<u^*$, this contradicts the hypothesis that 
 $w(t,ut+z)\leq \e$. Hence $u^*$ is a lower bound on the wave speed. This concludes the proof of the first part of Lemma \thv(speed.1).
 
 {\bf Case 2:} $2\left(1-\tg\right)\left(1-\e\right)\leq\left(\sqrt{2}-u\right)^2$. 
 
 In this case we insert the lower bound from Lemma \thv(lem.small) into \eqv(speed.7) and set $L=0$. Hence,
 \bea\Eq(lisa.lan12)
 w(t,ut)&\geq & \eee^{(\tg(1-\e)-\tb) t}\sfrac {1-\tb/\tg}{\sqrt{2\pi t}}\int_{-\left((K+a)/2 \land b\right)}^0 \eee^{-\frac{(ut-y)^2}{2t}}
 \left(1-\eee^{-\frac{b^2}{2t}}\right)dy\nonumber\\
 &\geq &
 \sfrac{b^2}{2t} \eee^{(\tg(1-\e)-\tb) t}\sfrac {1-\tb/\tg}{\sqrt{2\pi t}}\eee^{-\frac{u^2t}{2}}\frac{1}{u}\left(1-\eee^{-\left((K+a)/2 \land b\right) u }\right),
 \eea
 which implies \eqv(lisa.lan14).
Note that on the exponential scale  \eqv(lisa.lan12) is
\be\Eq(lisa.lan13)
 \eee^{-\frac{u^2t}{2}+(\tg(1-\e)-\tb) t},
\ee
implying \eqv(lisa.lan15). This finishes the proof of Lemma \eqv(speed.1).
   \end{proof}

\subsection{Precise control at the tip of the wave}
The estimates obtained on $w(t,x)$ allow for a finer control of the position of the 
wave as a function of $t$. We assume now that $a$ is a constant.
The most serious error term in the bounds comes from the $t O(\e)$ in the exponents. To obtain an error of order $1$, we want to choose $\e =O(1/t)$ in Lemmas \thv(speed.10) and \thv(speed.1). To this end, one needs to choose $K$ large enough, such that the terms on the right hand side of \eqv(speed.15) and \eqv(speed.5) are of order $1/t$. This requires to choose $K\sim c\ln(t)$. We state precise estimates only for the more interesting case $u^*>\sqrt{2\left(\tg-\tb\right)}$, analogous results in the other cases can be obtained in the same way.

\begin{lemma}\TH(fine.1) Assume that $\tb$ and $\tg$ are such that $u^*>\sqrt{2\left(\tg-\tb\right)}$.
Let $a$ be  a constant independent of $t$. Let   $c_+=1/(2-\sqrt 2)$. Then,  there exists a constant $0<C<\infty$, independent of $t$, such that, for $x=u^*+z$, 
\be\Eq(fine.2)
 w(t,u^*t+z)
 \leq C\frac {1}{u^* }t^{ c_+\sqrt{1-\tg}-1/2}|\ln t| \eee^{-z\sqrt 2\left(1-\sqrt{1-\tg}\right)}.
 \ee
 In particular, $w(t,u^*t+z)\leq C/t$ if, for some $\d>0$,  
 \be
 \Eq(fine.3)
 z\geq  z_+\equiv  \sfrac {c_+\sqrt{1-\tg}+1/2+\d}{\sqrt 2\left(1-\sqrt{1-\tg}\right)}{\ln t}.
 \ee
 \end{lemma}
 \begin{proof} The proof is straightforward from \eqv(speed.17), choosing $K=-c_+\ln t$. \end{proof}
 
 The next lemma shows that this  upper bound is not too bad.
 
 \begin{lemma}\TH(fine.4)
Assume that $\tb$ and $\tg$ are such that $u^*>\sqrt{2\left(\tg-\tb\right)}$. Let $a$ be  a constant independent of $t$. Then, with 
$c_-=\sqrt 2$, there exists a constant $0<C<\infty$, independent of $t$, such that, for $x=u^*t+z$,  $w(t,u^*t+z)\leq C/t$, then 
\be\Eq(fine.2.1)
 w(t,u^*t+z)
 \geq C\frac {1}{u^* }t^{- c_-  \sqrt{2(1-\tg)}(1+\sqrt 2)-1/2}                                                                                                                                                                                                                                                                                                                                                                          |\ln t|^2 \eee^{-z\sqrt 2\left(1-\sqrt{1-\tg}\right)}.
 \ee
 In particular, this is in contradiction with the assumption if 
 \be
 z\leq z_-\equiv -\sfrac{c_-\sqrt{2(1-\tg)}(1+\sqrt 2)-1/2}{\sqrt 2(1-\sqrt{1-\tg})}\ln t.
 \ee
 \end{lemma}
 \begin{proof} This is straightforward from \eqv(speed.5), choosing $K=c_-\ln t$.\end{proof}
 
 Lemma \thv(fine.4) tells us that $w(t, u^* t +z_-)$ is greater than $O(1/t)$. The next lemma tells us that at time of order $\ln t$ later, this will have grown to $O(1)$.

 \begin{lemma}\TH(lem.01)
Let $\e(t)>0$.  Assume that
 \be\Eq(lisa.lan16)
 w(t,ut+z)\geq \e(t) \quad \forall z\leq 0.
 \ee
 Then, for all $\d>0$ sufficiently small, there exists a constant $c$ such that
 \be\Eq(lisa.lan17)
 w\left(t+c\log \left(\e(t)^{-1}\right),ut+z\right)\geq \d \left(1-\frac{\tb}{\tg}\right)>0.
 \ee
 \end{lemma}
 \begin{proof}
 Starting from the Feynman-Kac formula, we have
 \bea\Eq(lisa.lan18)
 w(t+s,x)& =&\E_x\left[\eee^{\int_0^s \left(1-\tb-(1-\tg) \o\left(B_r-\sqrt{2}(s-r)\right)-\tg w\left(t+s-r,B_r\right)\right)dr}w\left(t,B_s\right)\right]\nonumber\\
 &=& \int_{-\infty}^{\infty} \frac{\eee^{-\frac{(x-y)^2}{2s}}}{\sqrt{2\pi s}}
 \E\left[\eee^{\int_0^s \left(1-\tb-(1-\tg) \o\left(x+\zet^s_{0,x-y}(r)-\sqrt{2}(s-r)\right)-\tg w\left(t+s-r,x+\zet^s_{0,x-y}(r)\right)\right)dr}w(t,y)\right]\nonumber\\
 &\geq& 
 \e(t) \int_{ut-1}^{ut} \frac{\eee^{-\frac{(x-y)^2}{2s}}}{\sqrt{2\pi s}}
\E\left[ \eee^{(\tg-\tb)s-\tg \int_0^s w\left(t+s-r,x+\frac{y-x}{s}r+\zet_{0,0}^s(r)\right)dr}\right],
 \eea
 where we used that $\o\leq 1$. Plugging in $x=ut+z$ and restricting the Brownian bridge to be larger than $-b$ for some $b>0$, we get that \eqv(lisa.lan18) is bounded from below by
 \be\Eq(lisa.lan19)
  \e(t) \int_{ut-1}^{ut} \frac{\eee^{-\frac{(ut+z-y)^2}{2s}}}{\sqrt{2\pi s}}
\E\left[ \eee^{(\tg-\tb)s-\tg \int_0^s w\left(t+s-r,(ut +z)\frac{s-r}{s}+\frac{y}{s}r+\zet_{0,0}^s(r)\right)dr}\1_{U_s^b=0} \right].
 \ee
 Now we assume that for all $0\leq r\leq s$,
 \be\Eq(lisa.lan20)
 w(t+s-r, ut+\tilde z)<\d \left(1-\frac{\tb}{\tg}\right), \quad \forall z-b <\tilde z<-b.
 \ee
 Then \eqv(lisa.lan19) is bounded from below by
 \bea\Eq(lisa.lan21)
&&   \e(t) \int_{ut-1}^{ut} \frac{\eee^{-\frac{(ut+z-y)^2}{2s}}}{\sqrt{2\pi s}}
 \eee^{(\tg-\tb)s-\tg\d \left(1-\frac{\tb}{\tg}\right)}  \P\left[ U_s^b =0 \right]\nonumber\\
&\geq&
 \e(t)  \frac{\eee^{-\frac{z^2}{2s}}}{\sqrt{2\pi s}}
 \eee^{(1-\d)(\tg-\tb)s} \left(1-\eee^{-b^2/2s}\right), 
 \eea
 where we used \eqv(lisa.lan11) to bound $\P\left[ U_s^b =0 \right]$. Let 
 $s=c\log\left(\e(t)^{-1}\right)$, 
 \be\Eq(lisa.lan22)
  \frac{\eee^{\frac{z^2}{2c\log\left(\e(t)\right)}}}{\sqrt{2\pi c\log\left(\e(t)^{-1}\right) }}
 \left(\e(t)\right)^{1-(1-\d)(\tg-\tb)c} \left(1-\eee^{\frac{b^2}{2c\log\left(\e(t)\right)}}\right).
 \ee
 Choosing $c$ large enough, \eqv(lisa.lan22) contradicts Assumption  \eqv(lisa.lan20). Hence, the claim of the lemma follows.
 \end{proof}
 
\begin{proof}[Proof of Theorem \thv(main)] Theorem \thv(main) follows directly from Lemmata \thv(fine.1), \thv(fine.4),  and \thv(lisa.lan17) in the case $u_c>\sqrt{2\left(\tg-\tb\right)}$. The analogous results when  $u_c=\sqrt{2\left(\tg-\tb\right)}$ are left to the reader.
 \end{proof}

\section{Discussion} 
In this paper we have used the Feynman-Kac representation to derive the 
speed of advance of a hitch-hiking subpopulation within an advancing population.
Apart from the fact that this allowed fairly sharp control of the precise behaviour of 
the wave fronts, the method provides a very clear intuitive understanding of the reason
for the acceleration in an advancing population compared to a fully established
one. Namely, the acceleration is driven by rare excursion of a Brownian bridge reaching ahead of the $B$-population. Translating this back into an underlying individual based model, heuristically this may be interpreted as having excursions of $A$ particles into the 
empty space ahead of the bulk wave taking advantage of higher growth rate in the absence of competition. 

Technically, we took advantage of the special features of the model that allowed to 
reduce the analysis to that of a scalar F-KPP equation with time-dependent parameters.
This is a delicate property that gets spoiled already if the diffusion coefficients of the 
two types are different. An explicit useable Feynman-Kac representation 
for systems of pdes does not exist. Still, we are optimistic that the Feynman-Kac representation (used for each one-dimensional component of the system) can be used 
in such situations. This is subject of ongoing research.

\appendix

\section{The Laplace method with prefactor}

To evaluate the asymptotic behaviour of integrals, we use Laplace's method. Below we state for convenience the results we use. Proofs can be found in many places in the literature, e.g. \cite{fedoryuk}.

 \begin{lemma}\TH(lem.laplace)
 Let $f:(0,1)\to \R$ be a twice differentiable function with a unique maximiser at $s^*$ and $f''(s^*)<0$. Moreover, let $P(s,t)$ be a rational function in $s$ and $t$. 
 Then 
 \be\Eq(lap.1)
\lim_{t\uparrow\infty}\frac{ \sqrt{-tf''(s^*)} \int_0^1 P(s,t) \eee^{t f(s)} ds}{\sqrt{2 \pi} P(s^*,t) \eee^{tf(s^*)} }=1.
 \ee
 \end{lemma}

We also need a similar statement for integrals that get their contribution from the boundary of the domain of integration. 

\begin{lemma}\TH(lem.laplace2)
 Let $f:(y,1)\to \R$ be a twice differentiable function, which is monotone decreasing. Moreover, let $P(s,t)$ be a polynomially bounded  function in $s$ and $t$ in $[y,1]$. 
 If, for some $\d\geq 0$,  $\lim_{x\downarrow y}(y-x)^{-\d} P(y,t) =c(t,y)$, then
 \be
\lim_{t\uparrow\infty}(tf'(y))^{1+\d} \eee^{t f(y)}\int_y^1 \frac{P(s,t)}{c(t,y)} \eee^{-t f(s)} ds = \G(1+\d).
\ee
 \end{lemma}
%

%
%
%

\end{document}